\documentclass[a4paper]{article}

\usepackage[a4paper,top=3cm,bottom=2cm,left=3cm,right=3cm,marginparwidth=1.75cm]{geometry}

\usepackage{url}
\usepackage{amsmath, amsthm, amsfonts, amssymb}
\usepackage{graphicx}
\usepackage[colorinlistoftodos]{todonotes}
\usepackage[colorlinks=true, allcolors=blue]{hyperref}
\usepackage{mathrsfs}

\newcommand\blfootnote[1]{
  \begingroup
  \renewcommand\thefootnote{}\footnote{#1}
  \addtocounter{footnote}{-1}
  \endgroup
}

\newcommand{\call}[1]{\mathcal{#1}}
\newcommand{\double}[1]{\mathbb{#1}}
\newcommand{\C}{\double{C}}

\newcommand{\N}{\double{N}}

\newcommand{\Q}{\double{Q}}

\newcommand{\hol}{\call{O}}

\newcommand{\D}{\mathscr{D}}

\newcommand{\mm}{\mathfrak{m}}

\newcommand{\wtilde}[1]{\widetilde{#1}}
\newcommand{\what}[1]{\widehat{#1}}

\DeclareMathOperator{\Rees}{\mathcal{R}}

\DeclareMathOperator{\gr}{gr}

\DeclareMathOperator{\rank}{rank}
\DeclareMathOperator{\Sing}{Sing}
\DeclareMathOperator{\Der}{Der}

\DeclareMathOperator{\diag}{diag}
\DeclareMathOperator{\Fitt}{Fitt}

\DeclareMathOperator{\htt}{ht}
\DeclareMathOperator{\tr}{tr}

\theoremstyle{plain}
\newtheorem{Th}{Theorem}[section]
\newtheorem{Prop}[Th]{Proposition}
\newtheorem{Cor}[Th]{Corollary}
\newtheorem{Lem}[Th]{Lemma}
\newtheorem{Ex}[Th]{Example}
\newtheorem{Conj}[Th]{Conjecture}
\theoremstyle{definition}
\newtheorem{Def}[Th]{Definition}
\newtheorem{Rmk}[Th]{Remark}

\begin{document}

\title{\textbf{On strong Euler-homogeneity and Saito-holonomicity for complex hypersurfaces. Applications to a conjecture on free divisors}}
\author{Abraham del Valle Rodríguez\footnote{\noindent ORCID number 0009-0008-5613-6333. Departamento de Álgebra and IMUS, Universidad de Sevilla, C/Tarfia s/n, 41012 Sevilla (Spain). \\
Email: adelvalle2@us.es.}}
\date{}

\maketitle

\begin{abstract}

We first develop some criteria for a general divisor to be strongly Euler-homogeneous in terms of the Fitting ideals of certain modules. We also study new variants of Saito-holonomicity, generalizing Koszul-free type properties and characterizing them in terms of the same Fitting ideals.

Thanks to these advances, we are able to make progress in the understanding of a conjecture from 2002: a free divisor satisfying the Logarithmic Comparison Theorem (LCT) must be strongly Euler-homogeneous. Previously, it was known to be true only for ambient dimension $n \leq 3$ or assuming Koszul-freeness. We prove it in the following new cases: assuming strong Euler-homogeneity outside a discrete set of points; assuming the divisor is weakly Koszul-free; for $n=4$; for linear free divisors in $n=5$.

Finally, we refute a conjecture stating that all linear free divisors satisfy LCT, are strongly Euler-homogeneous and have $b$-functions with symmetric roots about $-1$.

\blfootnote{\noindent 2020 \emph{Mathematics Subject Classification.}
32S25, 32S05, 14F40. \\
\textbf{Keywords:} free divisor, Logarithmic Comparison Theorem, logarithmic vector field, Euler-homogeneity, Saito-holonomicity. \\
\textbf{Declarations:} The author is supported by a Junta de Andalucía PIF fellowship num. PREDOC\_00485, by PID2020-114613GB-I00 and by PID2024-156912NB-I00. He has no further relevant interests to disclose.}
\end{abstract}

\section{Introduction}
\label{SectionIntro}

Let $X$ be a complex analytic manifold of dimension $n$ and let $D$ be a divisor (i.e. a hypersurface; we will only consider reduced divisors and reduced local equations of them). Let $U = X \setminus D$ be the complement of the divisor $D$ in $X$ and let $j: U \to X$ be the inclusion. Let $\Omega^\bullet_X(*D)$ be the \emph{meromorphic de Rham complex}, the complex of sheaves (of $\C$-vector spaces) of meromorphic differential forms with poles along $D$, together with the exterior differential.

Grothendieck's Comparison Theorem \cite[Theorem 2]{Groth} states that the natural morphism

$$ \Omega^\bullet_X(*D) \to Rj_* \C_U $$

\noindent is an isomorphism in the derived category of sheaves of $\C$-vector spaces. 

It is a general fact that the hypercohomology of $Rj_* \C_U$ is isomorphic to $H^*(U; \C)$, the singular cohomology of $U$. Thus, Grothendieck's Comparison Theorem tells us that the meromorphic de Rham complex can be used to compute the cohomology of the complement of a divisor. However, the sheaves of $\Omega^\bullet_X(*D)$ are not coherent $\hol_X$-modules, so it is interesting to see when one can replace the meromorphic de Rham complex by a quasi-isomorphic complex whose objects are coherent.

K. Saito introduced in \cite{Saito80} \emph{logarithmic forms (along $D$)} as those meromorphic forms $\omega$ such that $f \omega$ and $f d\omega$ are holomorphic, where $f$ is a local equation of $D$. Logarithmic forms constitute a complex of sheaves of $\C$-vector spaces with the exterior differential (in fact, a subcomplex of $\Omega_X^\bullet(*D)$), which is called the \emph{logarithmic de Rham complex} and it is denoted by $\Omega_X^\bullet(\log D)$. As the $\Omega^i_X(\log D)$ are coherent $\hol_X$-modules, it is natural to ask when the inclusion $\Omega^\bullet_X(\log D) \hookrightarrow \Omega^\bullet_X(*D)$ is a quasi-isomorphism, so that $\Omega^\bullet_X(\log D)$ could serve our purpose. This question leads us, by analogy with Grothendieck's Comparison Theorem, to the following definition:

\begin{Def}
    We say that $D$ satisfies the \emph{Logarithmic Comparison Theorem (LCT)} if the inclusion

    \begin{equation}
    \label{EqLCTMapQuasiIsomIncl}
    \Omega^\bullet_X(\log D) \hookrightarrow \Omega^\bullet_X(*D)
    \end{equation}

    \noindent is a quasi-isomorphism.
\end{Def}

If LCT holds for $D$, in particular, the inclusion $\Omega^\bullet_{X,p}(\log D) \hookrightarrow \Omega^\bullet_{X,p}(*D)$ is a quasi-isomorphism for all $p \in D$. However, it is not clear that $\Omega^\bullet_{X,p}(\log D) \hookrightarrow \Omega^\bullet_{X,p}(*D)$ being a quasi-isomorphism implies that LCT holds in a neighbourhood of $p$. Although we do not know of any counterexample, we tend to think that this is not true and Example \ref{ExQuasiAt0ButNotLCT} supports this idea. Therefore, we will say that \emph{$D$ satisfies LCT at $p$} or that \emph{the germ $(D,p)$ verifies LCT} if $D$ satisfies LCT in a neighbourhood of $p$.

P. Deligne proved that LCT holds for normal crossing divisors \cite[Ch. II, Lemme 6.9]{Deligne70}, a fact that plays a crucial role in the development of mixed Hodge theory. Later, F.J. Castro Jiménez, L. Narváez Macarro and D. Mond showed in \cite{CNM96} that LCT also holds for any locally quasihomogeneous free divisor. H. Terao had conjectured in 1977 that any hyperplane arrangement (which are obviously locally quasihomogeneous), free or not, satisfies LCT. This was finally proved in 2024 by D. Bath \cite{BathHypArr}. For a recent detailed survey on LCT, see \cite{SurveyLCT}.

In 2002, F.J. Calderón Moreno, D. Mond, L. Narváez Macarro and F.J. Castro Jiménez proved that LCT is equivalent to local quasihomogeneity for plane curves \cite[Theorem 1.3]{2002}. They also showed that this is not true in higher dimensions: the divisor given by the equation $x_1x_2(x_1+x_2)(x_1+x_2x_3) = 0$ in $\C^3$, usually called \textit{the four lines}, verifies LCT but is not locally quasihomogeneous \cite[Section 4]{2002}. Nevertheless, this divisor still has the property of \emph{strong Euler-homogeneity} (see Definition \ref{DefFEH}). This is a weaker property that is equivalent to local quasihomogeneity for isolated singularities. A more detailed explanation is given below.

In fact, all known examples of divisors satisfying LCT are strongly Euler-homogeneous. The same authors conjectured that this is always the case for a particular type of divisors: \emph{free divisors}. 

\begin{Conj}
\label{ConjLCT-SEH}
    If $D$ is a free divisor in a complex analytic manifold $X$ of dimension $n$ that satisfies the Logarithmic Comparison Theorem, then it is strongly Euler-homogeneous.
\end{Conj}

Very little is known about the veracity of this conjecture. Apart from plane curves, for which both properties are equivalent, as indicated above, Conjecture \ref{ConjLCT-SEH} is also known to be true for $n=3$. M. Granger and M. Schulze proved it in \cite[Theorem 1.6]{GS} with the help of the formal structure theorem for logarithmic vector fields. Later, in 2015, L. Narváez Macarro showed that, for Koszul-free divisors, both properties (and some others) are also equivalent \cite[Theorem 4.7]{NarvDual} (see also \cite[Corollary 1.8]{Tor04}). In the non-free setting, recent work by D. Bath and M. Saito confirmed that LCT also implies local quasihomogeneity for isolated singularities \cite[Corollary 1]{BathSaitoTLCT}.

One might also ask for the converse. D. Bath and M. Saito also characterized when a quasihomogeneous isolated singularity verifies LCT in terms of the unipotent monodromy part of the vanishing cohomology \cite[Remark 1.4a]{BathSaitoTLCT}. This provides several counterexamples to the reverse implication, at least for non-free divisors. However, as far as we are concerned, no free counterexample is known.

For free divisors, LCT is equivalent to the canonical map

\begin{equation}
\label{EqLCTMapQuasiIsomJ}
    j_! \C_U \to \Omega^\bullet_X(\log D)(\hol_X(-D))
\end{equation}

\noindent being a quasi-isomorphism (that is, the complex $\Omega^\bullet_X(\log D)(\hol_X(-D))$ is exact at any $p \in D$)  \cite[Corollary 1.7.2]{NarvLinearJacIdeal}. The following example shows that it is possible that (\ref{EqLCTMapQuasiIsomJ}) is a quasi-isomorphism at a point without being so in a neighbourhood:

\begin{Ex}
\label{ExQuasiAt0ButNotLCT}
    The free divisor $D=V(f) \subset (\C^3,0)$ with $f=xyz(x^5 z + x^3 y^3 + y^5 z)$ verifies that (\ref{EqLCTMapQuasiIsomJ}) is a quasi-isomorphism at $0$ (due to homogeneity, see \cite[Remark 1.7.4]{NarvLinearJacIdeal}) but not in a neighbourhood. If this were the case, then LCT would hold for $D$ and so $D$ would be strongly Euler-homogeneous (recall that Conjecture \ref{ConjLCT-SEH} is true for $n=3$). But an easy computation (using, for instance, Theorem \ref{ThGeomEH}) shows that $D$ is not strongly Euler-homogeneous at any point $(0,0,z)$ with $z \neq 0$. However, we do not know if (\ref{EqLCTMapQuasiIsomIncl}) is a quasi-isomorphism at $0$ or not.
\end{Ex}

The original motivation behind this article was the study of Conjecture \ref{ConjLCT-SEH}, which we prove to be true in some new cases. To achieve this goal, we first need to advance in the comprehension of strong Euler-homogeneity, for which we are able to give some characterizations.

Let $\Der_X$ be the $\hol_X$-module of $\C$-derivations (or vector fields) of $\hol_X$. We say that a germ $f \in \hol_{X,p}$ is \emph{Euler-homogeneous} if it satisfies an equation of the form $\delta(f) = f$ for some $\delta = \sum_{i=1}^n a_i \partial_i \in \Der_{X,p}$. Such a derivation is called an \emph{Euler derivation} or an \emph{Euler vector field}.

All quasihomogeneous polynomials are Euler-homogeneous due to Euler's theorem but, of course, there are Euler-homogeneous functions that are not polynomials (take, for example, $f=e^{x_1}$ and $\delta = \partial_1$). However, K. Saito proved in 1971 that, for isolated singularities, quasihomogeneous polynomials are in essence all the examples. Precisely, he proved that, if $f$ defines an isolated singularity in a neighbourhood of $0$, then $f$ is Euler-homogeneous if and only if there exists a coordinate system in which $f$ becomes a quasihomogeneous polynomial \cite[Theorem 4.1]{Saito71}. For non-isolated singularities, this is no longer true, even for polynomials. As explained before, \emph{the four lines} is a counterexample.

When studying Euler-homogeneity of local equations of a divisor $D \subset X$ at a point $p \in D$, it is desirable that this property does not depend on the particular choice of local equation, which is guaranteed if we ask the Euler derivation to vanish at $p$. This is \emph{strong Euler-homogeneity}:

\begin{Def}
\label{DefFEH}
A germ of holomorphic function $f \in \hol_{X,p}$ is called \emph{strongly Euler-homogeneous at $p$} if there exists a germ of derivation $\delta$ vanishing at $p$ (i.e. $\delta \in \mm_{X,p} \Der_{X,p}$, where $\mm_{X,p}$ denotes the maximal ideal of the local ring $\hol_{X,p}$) such that $\delta(f) = f$. A divisor $D$ is said to be \emph{strongly Euler-homogeneous at $p \in D$} if some (or any) reduced local equation of $D$ at $p$ is strongly Euler-homogeneous and it is called \emph{strongly Euler-homogeneous on a subset $E \subset D$} if it is so at any $p \in E$. When $E=D$, we simply say that $D$ is \emph{strongly Euler-homogeneous}. 
\end{Def}

In order to study this property, \emph{logarithmic derivations (or vector fields) along $D$} turn out to be very useful. This concept was extensively studied by K. Saito in \cite{Saito80} and refers to those derivations leaving invariant the defining ideal of $D$. These derivations form an $\hol_X$-module, denoted by $\Der_X(-\log D)$. Locally, a derivation $\delta \in \Der_{X,p}$ belongs to $\Der_{X,p}(-\log D)$ if $\delta(f) \in (f)$ for some (or any) reduced local equation $f$ of $D$ at $p$. In this case, we will also say that $\delta$ is a \textit{logarithmic derivation for $f$}.

Consider a reduced local equation $f$ of $D$ at $p$ and a generating set $\mathcal{S}=\{\delta_1, \ldots, \delta_m\}$ of $\Der_{X,p}(-\log D)$. Once a local coordinate system $x_1, \ldots, x_n$ is chosen, we define the \emph{Saito matrix with respect to $\mathcal{S}$} as $A = (\delta_i(x_j))_{i,j}$, the $m \times n$ matrix whose entries are the coefficients in the expression of $\delta_1, \ldots, \delta_m$ as linear combinations of $\partial_1, \ldots, \partial_n$. Writing $\overline{\alpha}:=(\alpha_1, \ldots, \alpha_m)^t$, where $\alpha_i \in \hol_{X,p}$ is such that $\delta_i(f) = \alpha_i f$, we define the \emph{extended Saito matrix with respect to $\mathcal{S}$ and $f$} as the $m \times (n+1)$ matrix $\tilde{A} = (A \mid -\overline{\alpha})$. Strictly speaking, these matrices also depend on the coordinate system, but we will omit this for the sake of brevity.

When $\Der_X(-\log D)$ is a locally free $\hol_X$-module, we say that $D$ is a \emph{free divisor}. If $D$ is free, then $\Der_{X,p}(-\log D)$ is free of rank $n$ as an $\hol_{X,p}$-module for every $p \in X$. By Saito's criterion \cite[Theorem 1.8]{Saito80}, the freeness of $\Der_{X,p}(-\log D)$ is equivalent to the existence of some $\delta_1, \ldots, \delta_n \in \Der_{X,p} (-\log D)$ such that if $f \in \hol_{X,p}$ is a reduced local equation at $p$, then there exists a unit $u \in \hol_{X,p}$ with $f = u \det(A)$, where $A$ is the Saito matrix with respect to $\{\delta_1, \ldots, \delta_n\}$. In this case, $\{\delta_1, \ldots, \delta_n\}$ is automatically a basis of $\Der_{X,p}(-\log D)$ and we will also say that $f$ is \emph{free}. This close relation between the equation and logarithmic derivations makes free divisors easier to work with.

After setting some preliminary results in Section \ref{SectionPrelim}, we begin Section \ref{SectionCriterionSEH} by defining some analytic closed subsets $D_i \subset X$ (resp. $\tilde{D}_i \subset X$) as the vanishing locus of certain Fitting ideals and seeing that they can be explicitly described as the set of points at which the rank of $A$ (resp. $\tilde{A}$) is less or equal than $i$. This allows us to give a criterion for strong Euler-homogeneity in a neighbourhood of a point $p$ in terms of these sets. We also give the algebraic translation of this criterion and a necessary condition for a divisor to be strongly Euler-homogeneous outside $D_0$. 

K. Saito defined in \cite[Definition 3.8]{Saito80} \emph{holonomic divisors} as those for which, at each point, the logarithmic stratification is locally finite and characterized them in terms of the dimensions of the sets $D_i$ (he called them $A_r$). Later, Calderón-Moreno defined in \cite[Definition 4.1.1]{Cald1999} the notion of \emph{Koszul-freeness} for free divisors in terms of regular sequences in the graded ring $\gr \D_X$ (where $\D_X$ denotes, as usual, the sheaf of differential operators on $X$). Although this was not initially noticed, these two notions turned out to be the same for free divisors \cite[Theorem 7.4]{GMNS}. In particular, Koszul-freeness can also be characterized in terms of these sets.

Still in the free setting, in \cite[Definition 1.13]{NarvDual} and \cite[Remark 7.3]{SymmetryGS}, \emph{weak} and \emph{strong Koszul-freeness} were introduced (see Definition \ref{DefKoszul}). While studying these properties, we asked ourselves if they could be extended to non-free divisors. This is what led us to define, by analogy, weak and strong Saito-holonomicity (Definition \ref{DefSaitoHolonomic}). And we found out that they have a characterization in terms of the dimensions of the sets $\tilde{D}_i$ (Theorem \ref{ThCaractSH}). This comprises the first half of Section \ref{SectionSaitoHolonomic}.

The rest of Section \ref{SectionSaitoHolonomic} is devoted to study these new versions of Saito-holonomicity and their relation with strong Euler-homogeneity and another interesting property: being of linear Jacobian type. In \cite[Proposition 1.11 and Theorem 4.7]{NarvDual}, it was shown that being strongly Koszul-free and being of linear Jacobian type are equivalent for free divisors and that, under Koszul-freeness hypothesis, strong Koszul-freeness is equivalent to strong Euler-homogeneity (among other properties). We use the developed criteria for strong Euler-homogeneity and (strong) Saito-holonomicity to generalize these results where possible.

Formal coordinate changes are much more flexible and they permit a particularly nice description of the module of formal logarithmic derivations. Thus, it is interesting to extend the criteria developed in Sections \ref{SectionCriterionSEH} and \ref{SectionSaitoHolonomic} to the ring of formal power series. This is done in Section \ref{SectionConvToFormal}, whose results will be very helpful in the following section.

In Section \ref{SectionConj}, we prove Conjecture \ref{ConjLCT-SEH} in some new cases. It is divided into three subsections. In Subsection \ref{SubsectionPuncturedBall} we study free divisors that are strongly Euler-homogeneous in a punctured neighbourhood of a point, characterizing when they are strongly Euler-homogeneous at the given point in terms of the existence of non-topologically nilpotent singular logarithmic derivations (Theorem \ref{ThSEHoutside0}). As a consequence, we prove a weak version of the conjecture in arbitrary dimension: if $D$ is a germ of free divisor in $(\C^n,0)$ satisfying LCT and being strongly Euler-homogeneous outside a discrete set of points, then it is so everywhere (Theorem \ref{ThConjOutsideDiscreteSet}). In Subsection \ref{SubsectionWK}, we use this result to prove that Conjecture \ref{ConjLCT-SEH} also holds in arbitrary dimension for weakly Koszul-free divisors (Theorem \ref{ThConjWKoszul}). Subsection \ref{SubsectionDim4} presents the proof of the conjecture for the $4$-dimensional case (Theorem \ref{ThConjDim4}), after developing an intrinsic version of the formal structure theorem (Theorem \ref{ThFST}).

Finally, M. Granger, D. Mond, A. Nieto and M. Schulze posed a conjecture in \cite{GMNS} stating that all linear free divisors satisfy LCT and are strongly Euler-homogeneous, and they proved it for ambient dimension $n \leq 4$. M. Granger and M. Schulze also conjectured in \cite[Conjecture 1.5]{SymmetryGS} that all linear free divisors have $b$-functions with symmetric roots about $-1$. We present in Section \ref{SectionLFD} a counterexample in $n=5$ for these two conjectures. Nonetheless, from the previous results, we are also able to show that Conjecture \ref{ConjLCT-SEH} holds for linear free divisors in ambient dimension $n=5$. \\

\section{Preliminaries}
\label{SectionPrelim}

Within this section, we fix the notation and sum up the main results of \cite[Sections 2 and 3]{dVR24}, which are going to be used throughout the rest of the text. Due to their different nature, we have divided them into three subsections.

Since the properties we are dealing with are local, it will be enough to consider a germ of free divisor $D=V(f)$ in $(\C^n, 0)$. Let $(\hol, \mm):= (\hol_{\C^n,0}, \mm_{\C^n,0})$, which can be identified with the convergent power series ring $\C\{x_1, \ldots, x_n\}$. Let $\Der := \Der_{\C^n,0}$ be the $\hol$-module of $\C$-derivations of $\hol$, which is free with a basis given by the partial derivatives $\{\partial_1, \ldots, \partial_n\}$ and let $\Der_f := \Der_{\C^n,0}(-\log D)$ be the $\hol$-module of logarithmic derivations for $f$. Recall that $\Der_g = \Der_f$ if $g$ is another reduced local equation of $D$ (i.e. $g=uf$ for some unit $u$).

\subsection{About products with smooth factors}

An important concept that will be useful is that of a \emph{product} (\emph{with a smooth factor}, but we will omit this for the sake of brevity):

\begin{Def}
    Let $(D,p)$ be a germ of a divisor in a complex analytic manifold $X$ of dimension $n$. We will say that $(D,p)$ is a \emph{product with a smooth factor} (from now on, just \emph{product}) if there exists some germ of divisor $(D', 0) \subset (\C^{n-1},0)$ such that $(D,p)$ is biholomorphic to $(D',0) \times (\C,0)$. In other words, if $f \in \hol_{X,p}$ is a reduced local equation of $D$ at $p$, then there exists a coordinate change such that $f=ug$ for some unit $u$ and some convergent power series $g$ in $n-1$ variables. In this case, we will also say that $f$ is a \emph{product (with a smooth factor) at $p$}.
\end{Def}

The following is an equivalent algebraic condition:

\begin{Lem}
\label{LemAlgProd}
Let $f$ be a reduced local equation of a divisor $D \subset X$ at a point $p \in D$. Then, $f$ is a product at $p$ if and only if it admits a non-singular logarithmic derivation, that is, if $\Der_{X,p}(-\log D)$ is not contained in $\mm_p \Der_{X,p}$. 
\end{Lem}

\begin{proof}
    It can be easily deduced from \cite[Lemma 3.5]{Saito80}.
    
\end{proof}

The next lemma (see \cite[Lemma 2.2]{CNM96} and \cite[Lemma 3.2]{GS}) will allow us, by applying an inductive argument, to just consider the case in which the local equation of the divisor is not a product:

\begin{Lem}
\label{LemProd}
Let $f \in \hol$ be reduced. Suppose $f$ is a product, so that there exists $g \in \hol$ depending only on $n-1$ variables in some coordinate system with $f=ug$ for some unit $u$. Then, freeness and strong Euler-homogeneity are equivalent for $f$ and $g$. Moreover, the divisor defined by $f$ in $(\C^n,0)$ satisfies LCT if and only if the divisor defined by $g$ in $(\C^{n-1},0)$ does.
\end{Lem}

\subsection{From convergent to formal definitions}

If we see convergent power series as formal power series we have a wider range of coordinate changes, which gives us more flexibility. Thus, we will also consider $\what{\hol} = \C[[x_1, \ldots, x_n]]$, the $\mm$-adic completion of $\hol$, which is the local ring of formal power series with maximal ideal $\what{\mm} = (x_1, \ldots, x_n)$.

Let us denote by $\what{\Der}$ the $\mm$-adic completion of $\Der$, which coincides with the $\what{\hol}$-module of $\C$-derivations of $\what{\hol}$ and is free with a basis given by the partial derivatives. As in the convergent case, for a formal power series $g \in \what{\hol}$, we will denote by $\Der_g$ the $\what{\hol}$-module of formal logarithmic derivations for $g$, formed by those $\delta \in \what{\Der}$ for which $\delta(g) \in (g)$. If we denote an element $f \in \hol$ by $\hat{f}$ when seen inside $\what{\hol}$, then by flatness the $\what{\hol}$-module $\Der_{\hat{f}}$ is precisely the $\mm$-adic completion of the $\hol$-module $\Der_f$ and we will denote it by $\what{\Der}_f$.

\begin{Rmk}
    As in the convergent case, if $g,h \in \what{\hol}$ are such that $h=ug$ for some formal unit $u$, then $\Der_h = \Der_g$.
\end{Rmk}

The notions of being a product and strongly Euler-homogeneous can be extended in a natural way to formal power series:

\begin{Def}
    Let $f \in \what{\hol}$. We say that $f$ is:

    \begin{itemize}
        \item[(1)] a \emph{product} if $\Der_f \not\subset \mm \what{\Der}$.
        \item[(2)] \emph{strongly Euler-homogeneous} if there exists a formal derivation $\delta \in \mm \what{\Der}$ such that $\delta(f)=f$.
    \end{itemize}

    We say that $f \in \hol$ is a \emph{formal product} or \emph{formally strongly Euler-homogeneous} if it verifies the corresponding property as an element of $\what{\hol}$.
\end{Def}

The next result (see \cite[Propositions 2.1, 2.2 and 2.3]{dVR24}) ensures that these definitions (and that of being reduced) coincide when convergent power series are seen as formal ones. Thus, we will just have to prove that $f$ satisfies the required property as a formal power series (which, in general, will be easier): 

\begin{Prop}
\label{PropForm}
    Let $f \in \hol$. Then,

    \begin{itemize}
        \item[(1)] $f$ is reduced in $\hol$ if and only it is reduced in $\what{\hol}$.
        
        \item[(2)] $f$ is strongly Euler-homogeneous if and only it is formally strongly Euler-homogeneous.
        
        \item[(3)] $f$ is a product if and only if it is a formal product.
    \end{itemize}
\end{Prop}

\subsection{About singular derivations}

Once a coordinate system is chosen, each derivation $\delta \in \Der$ can be uniquely decomposed as a sum $\delta = \sum_{i=-1}^\infty \delta_i$, where $\delta_i = \sum_{j=1}^n a_{ij} \partial_j$ with $a_{ij}$ being homogeneous of degree $i+1$ for all $j=1, \ldots, n$. Moreover, $\delta_0$, which is called the \emph{linear part of $\delta$}, can be written as $\underline{x} A \overline{\partial}$, where $A$ is a constant matrix, $\underline{x} = (x_1, \ldots, x_n)$ and $\overline{\partial} = (\partial_1, \ldots, \partial_n)^t$. The same is true for a formal derivation $\delta \in \what{\Der}$. We say $\delta \in \Der$ (resp. $\delta \in \what{\Der}$) is \emph{singular} if $\delta_{-1} = 0$ or, equivalently, $\delta \in \mm \Der$ (resp. $\delta \in \mm \what{\Der} = \what{\mm} \what{\Der}$).

We are going to make use of the Jordan-Chevalley decomposition for a singular formal derivation developed by R. Gérard and A. Levelt \cite[Théorèmes 1.5, 2.2 and 2.3]{GL}. To that end, we need to extend the classical concepts of semisimple and nilpotent endomorphisms of a finite dimensional vector space. This can be done in a proper way for some vector spaces of infinite dimension \cite[\S 1]{GL}. However, we will give here equivalent definitions that will be more practical for our purpose.

A formal singular derivation $\delta \in \mm\what{\Der}$ leaves invariant every $\what{\mm}^k$, so it induces a map in each quotient $\delta^{(k)}: \what{\hol}/\what{\mm}^k \to \what{\hol}/\what{\mm}^k$. 

\begin{Def}
\label{DefSN}
Let $\delta \in \mm\what{\Der}$. We say $\delta$ is:
    
\begin{itemize}
    \item[(1)] \textit{semisimple} if the induced map $\delta^{(k)}$ is diagonalizable for all $k \in \N$.
    \item[(2)] \textit{topologically nilpotent} if the induced map $\delta^{(k)}$ is nilpotent for all $k \in \N$.
\end{itemize}
\end{Def}

\begin{Th}[Jordan-Chevalley decomposition]
\label{ThChangeCoordDiag}
Let $\delta$ be a singular derivation of $\what{\hol}$. Then, there exist two unique commuting singular derivations $\delta_S$, $\delta_N$ such that $\delta_S$ is semisimple, $\delta_N$ is topologically nilpotent and $\delta = \delta_S + \delta_N$. Moreover, there exists a regular system of parameters of $\what{\hol}$ (i.e. a formal coordinate system) in which $\delta_S$ is diagonal (i.e. $\delta_S$ is of the form $\sum_{i=1}^n \lambda_i x_i \partial_i$).
\end{Th}

The following proposition collects some useful facts about semisimple and topologically nilpotent derivations (see \cite[Section 3 and Lemma 5.1]{dVR24}):

\begin{Prop}
\label{PropSNdVR24}
    Let $\delta \in \mm \what{\Der}$ and $f \in \what{\hol}$, $f \neq 0$. The following holds:

    \begin{itemize}
    
        \item[(1)] $\delta$ is topologically nilpotent if and only if the matrix of its linear part is nilpotent. In particular, if $\delta$ has no linear part, then it is topologically nilpotent.
        
        \item[(2)] If $\delta$ is topologically nilpotent and $\delta(f) = \alpha f$ for some $\alpha \in \what{\hol}$, then $\alpha \in \what{\mm}$.

        \item[(3)] If $\delta$ is logarithmic for $f$, then so are $\delta_S$ and $\delta_N$.
        
        \item[(4)] If $\delta$ is semisimple and $\delta(f) = \alpha f$ for some $\alpha \in \what{\hol}$, then there exists a formal unit $u$ such that, for $g=uf$, $\delta(g) = \alpha_0 g$, where $\alpha_0$ is the constant term of $\alpha$.
    \end{itemize}
\end{Prop}

We define the \textit{trace} of a singular (convergent or formal) derivation as the trace of the linear map induced in the quotient $\mm/\mm^2$ (or $\what{\mm}/\what{\mm}^2$). In coordinates, this is just the trace of the matrix of its linear part. A key point in the proof of the known cases of Conjecture \ref{ConjLCT-SEH} is the existence of singular derivations with non-zero trace when LCT holds. This will also be important in the proof of the new cases:

\begin{Th}
\label{ThZeroTrace}
    If $D$ is a germ of free divisor in $\C^n$ satisfying LCT and $f \in \hol$ is a reduced local equation of $D$ at $0$, then there exist singular derivations in $\Der_f$ (and also in $\what{\Der}_f$) with non-zero trace.
\end{Th}

\begin{proof}
    It can be deduced from \cite[Sections 2 and 3]{2002} and \cite[Lemma 7.5]{GS}. See \cite[Section 4]{dVR24} for a $\D$-module argument.
    
\end{proof}

\section{Strong Euler-homogeneity around a point}
\label{SectionCriterionSEH}

Let $D=V(f)$ be a germ of divisor in $(\C^n, 0)$ with $n \geq 2$. We first introduce some analytic sets that will be used throughout the rest of the text and will play a crucial role in the characterization of strong Euler-homogeneity and Saito-holonomicity properties.

Let $\mathcal{S}=\{\delta_1, \ldots, \delta_m\}$ be a generating set of $\Der_f$ (note that $m \geq n$ by coherence, since at a smooth point there cannot be less than $n$ generators). Let $A = (\delta_i(x_j))_{i,j}$ be the Saito matrix with respect to $\mathcal{S}$. Let $\alpha_i \in \hol$ be such that $\delta_i(f) = \alpha_i f$ for $i=1, \ldots, m$ and let $\tilde{A} = (A|-\overline{\alpha})$ be the extended Saito matrix with respect to $\mathcal{S}$ and $f$. 

For an $\hol$-module $M$, let $\Fitt_i(M)$ denote the $i$-th Fitting ideal of $M$. Consider the Jacobian ideal $\mathcal{J}_f = \hol(\partial_1(f), \ldots, \partial_n(f), f)$ and the $\hol$-module $\mathcal{J}_f/(f)$ generated by the classes of the partial derivatives $[\partial_1(f)], \ldots, [\partial_n(f)]$. We set $I_i := \Fitt_{n-i}\left(\mathcal{J}_f/(f)\right)$ and $D_i := V(I_{i+1})$.

The rows of $A$ are a generating set of the syzygies of $[\partial_1(f)], \ldots, [\partial_n(f)]$. A free presentation of $\mathcal{J}_f/(f)$ as an $\hol$-module is then:

$$ \hol^m \stackrel{\cdot A}{\longrightarrow} \hol^n \stackrel{\underline{e}_i \mapsto [\partial_i(f)]}{\xrightarrow{\hspace{1 cm}}} \mathcal{J}_f/(f) \to 0. $$

\vspace{0,3 cm}

Thus, $I_i$ is the ideal generated by the minors of order $i$ of $A$ for $1 \leq i \leq n$, $I_i = \hol$ for $i \leq 0$ and $I_i = 0$ for $i \geq n+1$. The set $D_i$ is then the vanishing locus of all minors of order $(i+1)$, that is, the set of those points at which $A$ has rank less or equal than $i$:

\vspace{0,2 cm}

$$ D_i = \{ p \in (\C^n,0) \mid \rank A(p) \leq i\}.$$

\vspace{0,1 cm}

Let us see some properties of these sets: 

\begin{Prop}
\label{PropD}
    The following properties hold:

    \begin{itemize}
        \item[(1)] $D_0 \subset D_1 \subset \ldots \subset D_{n-1} \subset D_n$. 
        \item[(2)] $D_n = (\C^n,0)$ and $D_{n-1} = D$. 
        \item[(3)] If $p \in D \setminus D_k$ for some $0 \leq k \leq n-2$, then $(D,p) \cong (D',0) \times (\C^{k+1},0)$, where $(D',0)$ is a germ of divisor in $(\C^{n-k-1},0)$.

        In particular, if $p \in D \setminus D_0$, then $(D,p) \cong (D',0) \times (\C,0)$ and $D$ is a product with a smooth factor at $p$.
    \end{itemize}

\end{Prop}

\begin{proof}

        As $I_{i+1} \subset I_i$ for all $i$, then $D_{i-1} = V(I_i) \subset V(I_{i+1}) = D_i$, which proves $(1)$. 
        
        Since $A$ is an $m \times n$ matrix with $m \geq n$, it is clear that $\rank A(p) \leq n$ for all $p$, so $D_n = (\C^n,0)$. By \cite[(1.5), iii)]{Saito80}, the inclusion $I_n \subset (f)$ holds, so $D = V(f) \subset V(I_n) = D_{n-1}$. The reverse inclusion is clear, since at a point $p$ outside $D$ all derivations are logarithmic and so the rank of $A(p)$, which is the dimension of the $\C$-vector subspace of $T_p(\C^n,0)$ generated by $\delta_1(p), \ldots, \delta_m(p)$, is $n$. This proves $(2)$.
        
        Let $p \in D \setminus D_k$, so $\rank A(p) > k$. Let us see by induction on $k$ that, in a neighbourhood of $p$, at least $k+1$ partial derivatives of $f$ belong to the ideal generated by $f$ and the rest of them.

        For $k=0$, at least one of the entries $a_{ij}$ (that corresponds to the coefficient of $\partial_j$ in $\delta_i$) of $A$ cannot vanish at $p$, so it is a unit in $\hol_{X,p}$. Thus, we can isolate $\partial_j (f)$ in $\delta_i(f) \in (f)$ and see that $\partial_j (f)$ belongs to the ideal generated by $f$ and $\partial_k (f)$, $k \neq j$.

        Suppose the result is true for $k \geq 0$ and let $\rank A(p) > k+1$. As before, we can find $i_0,j_0$ such that $a_{i_0j _0}(p) \neq 0$. Now set $\delta'_k = \delta_k-\dfrac{a_{kj_0}}{a_{i_0 j_0}} \delta_{i_0}$ for $k \neq i_0$ and $\delta'_{i_0} = \delta_{i_0}$. Let $A'$ be the Saito matrix with respect to $\{\delta'_1, \ldots, \delta'_m\}$ (which is still a generating set) and note that $a'_{k j_0}(p) = 0$ for $k \neq i_0$. Since $\rank A'(p) = \rank A(p) > k+1$ and $a'_{i_0 j_0}(p) = a_{i_0 j_0}(p) \neq 0$, the rank of $A'(p)$ is exactly $1 + \rank B(p)$, where $B$ is the submatrix of $A'$ formed by the entries $a'_{ij}$ with $i \neq i_0$, $j \neq j_0$. Thus, $\rank B(p) > k$ and, by induction hypothesis, at least $k+1$ partial derivatives of $f$ belong to the ideal generated by $f$ and the rest of them except $\partial_{j_0} (f)$. But, reasoning as in the case $k=0$, $\partial_{j_0} (f)$ also belongs to this ideal, so we conclude that at least $k+2$ partial derivatives of $f$ belong to the ideal generated by $f$ and the rest of them.
        
        Now, by \cite[Lemma 3.5]{Saito80}, there exist a unit $u$ and a local coordinate change such that $g=uf$ (which is another equation of $D$) depends only on $n-k-1$ variables. If we set $D' = V(g) \subset \C^{n-k-1}$, we get $(3)$. 
        
\end{proof}

\begin{Rmk}
    Recall that strong Euler-homogeneity is preserved under smooth factors. Therefore, if $p \in D \setminus D_0$, so that $(D,p) \cong (D',0) \times (\C,0)$, then $D$ is strongly Euler-homogeneous at $p$ if and only if so is $D'$ at $0$ \cite[Lemma 3.2]{GS}.
\end{Rmk}

Now, call $\tilde{I}_i := \Fitt_{n+1-i}\left(\mathcal{J}_f \right)$ and $\tilde{D}_i := V(\tilde{I}_{i+1})$. The rows of $\tilde{A}$ generate the syzygies of $\partial_1(f), \ldots, \partial_n(f), f$, so a free presentation of $\mathcal{J}_f$ as an $\hol$-module is:

$$ \hol^m \stackrel{\cdot \tilde{A}}{\longrightarrow} \hol^{n+1} \stackrel{\substack{\underline{e}_i \mapsto \partial_i(f), \hspace{0,1 cm} i=1, \ldots, n \\ \underline{e}_{n+1} \mapsto f}}{\xrightarrow{\hspace{2 cm}}} \mathcal{J}_f \to 0. $$

\vspace{0,3 cm}

Therefore, $\tilde{I}_i$ is the ideal generated by the minors of order $i$ of $\tilde{A}$ for $1 \leq i \leq n$, $\tilde{I}_i = \hol$ for $i \leq 0$ and $\tilde{I}_i = 0$ for $i \geq n+1$ (in the case $m>n$, the minors of order $n+1$ of $\tilde{A}$ are zero because there is always a non-trivial relation between any $n+1$ vector fields that trivially extends to the last column of $\tilde{A}$). As before, $\tilde{D}_i$ is the set of those points at which $\tilde{A}$ has rank less or equal than $i$:

\vspace{0,2 cm}

$$ \tilde{D}_i = \{p \in (\C^n,0) \mid \rank \tilde{A}(p) \leq i\}.$$

\vspace{0,2 cm}

\begin{Rmk}
$ $

\begin{itemize}
    \item[(a)] Note that, if $g$ is another reduced equation of $D$, then $\mathcal{J}_g = \mathcal{J}_f$ and $\mathcal{J}_f/(f) = \mathcal{J}_g/(g)$. Also, Fitting ideals do not depend on the choice of presentation. Thus, $D_i$ (resp. $\tilde{D}_i$) and $I_i$ (resp. $\tilde{I}_i$), which in principle might depend on the equation of $D$ or on the Saito matrix with respect to the chosen basis, depend only on $D$ itself.

    \item[(b)] Note that the ideal generated by the $\alpha_i$ can be written as the colon ideal $(\partial_1(f), \ldots, \partial_n(f)) : f$ (from now on, we will just write $\partial f : f$ for short). Whereas $I_i$ and $\tilde{I}_i$ do not depend on the choice of $f$, the ideal $\partial f : f$ does depend on such a choice. Indeed, if $g=uf$ for some unit $u$, then 
    
    \vspace{-0,2 cm}
    
    $$\partial g : g = (\alpha_1+\delta_1(u)u^{-1}, \ldots, \alpha_m + \delta_m(u)u^{-1}).$$

    \vspace{0,3 cm}
    
\end{itemize}
\end{Rmk}

The sets $\tilde{D}_i$ and ideals $\tilde{I}_i$ are related with the previous ones in the following way:

\begin{Prop}
\label{PropDIrel}
    The following properties hold:

    \begin{itemize}
        \item[(1)] $\tilde{I}_{i+1} \subset I_{i+1} + (\alpha_1, \ldots, \alpha_m) I_i \subset I_i \subset \tilde{I}_i$ \hspace{0,1 cm} for all $i=1, \ldots, n-1$.
        
        \item[(2)] $\tilde{D}_i \subset D_i \subset \tilde{D}_{i+1}$ \hspace{0,1 cm} for all $i=0, \ldots, n-2$.

        \item[(3)] $\tilde{D}_{n-1} = D_{n-2} = \Sing D$.
    \end{itemize}
\end{Prop}

\begin{proof}
    The minors of order $i+1$ of $\tilde{A}$ are the ones of $A$ together with those involving the last column $-\overline{\alpha}$, which are (expanding them by this column) linear combinations of minors of order $i$ of $A$ with coefficients in the $\alpha_i$. This justifies the first inclusion of $(1)$. The second inclusion is due to the fact that $I_{i+1} \subset I_i$ for all $i$. The last inclusion holds because, $A$ being a submatrix of $\tilde{A}$, the minors of order $i$ of $A$ are also minors of order $i$ of $\tilde{A}$. Since the map $V$ is inclusion-reversing, $(2)$ is immediate from $(1)$. 
    
    Let us prove $(3)$. Since $\dim (\Sing D) \leq n-2$ (or $\htt \mathcal{J}_f \geq 2$), by \cite[Theorem 5.1]{EisenbudFitting}, we have that $\sqrt{\mathcal{J}_f} =\sqrt{\Fitt_1(\mathcal{J}_f)} =  \sqrt{\tilde{I}_n}$. Thus,
    
    $$\tilde{D}_{n-1} = V(\tilde{I}_n) = V\left(\sqrt{\tilde{I}_n}\right) = V\left(\sqrt{\mathcal{J}_f}\right) = V(\mathcal{J}_f) = \Sing D.$$ 

    By $(2)$, $D_{n-2} \subset \tilde{D}_{n-1} = \Sing D$. Now, suppose that there exists $p \in \Sing D \setminus D_{n-2}$. Then, by $(3)$ in Proposition \ref{PropD}, $(D,p) \cong (\C^{n-1},0)$. But this is a contradiction because $ (\C^{n-1},0)$ is non-singular, so $D_{n-2} = \Sing D$ and we have $(3)$.        
    
\end{proof}

\begin{Rmk}
    Since $D_{n-1} = D$ and $\tilde{D}_{n-1} = \Sing D$, we have that $D_i$ and $\tilde{D}_i$ are closed analytic subsets of $D$ for $0 \leq i \leq n-1$.
\end{Rmk}

Now, we are going to give algebraic and geometric characterizations of strong Euler-homogeneity for points in a neighbourhood of $0$.

As $\tilde{D}_{n-1} = D_{n-2}$, we may wonder if it is always true that $\tilde{D}_i = D_{i-1}$ for all $i = 0, \ldots, n-1$ (where $D_{-1} = V(I_0) = V(\hol) = \varnothing$). We are going to see that this property is, in fact, equivalent to strong Euler-homogeneity in a neighbourhood of the origin.

This criterion is based on \cite[Lemma 7.5]{GMNS}, where the authors give a characterization of strong Euler-homogeneity for germs of free divisors that are strongly Euler-homogeneous at the origin. We have adapted this result to the case in which we do not know a priori whether there exists an Euler vector field. 

First, we give a characterization for strong Euler-homogeneity at a point $p$ in terms of the rank of the matrices $A(p)$ and $\tilde{A}(p)$:

\begin{Prop}
\label{PropRank}
    $D$ is strongly Euler-homogeneous at the point $p \in D$ if and only if $\rank \tilde{A}(p) = \rank A(p)+1$.
\end{Prop}

\begin{proof}
    Let us note that, as $\mathcal{S} = \{\delta_1, \ldots, \delta_m\}$ is a generating set of $\Der_f$, in a neighbourhood of the origin, $D$ is strongly Euler-homogeneous at $p$ if and only if there exists $a_1, \ldots, a_m \in \hol$ such that $\chi = \sum_{i=1}^m a_i \delta_i$ vanishes at $p$ and $\chi(f) = \gamma f$ with $\gamma(p) =\sum_{i=1}^m a_i(p) \alpha_i(p) \neq 0$. Moreover, the $\C$-vector space $\wtilde{N} = \{\bar{x} \in \C^m \mid \tilde{A}(p)^t \bar{x} = \bar{0}\}$ is always contained in the $\C$-vector space $N = \{\bar{x} \in \C^m \mid A(p)^t \bar{x} = \bar{0}\}$, with equality if and only if $\dim \wtilde{N} = m-\rank \tilde{A}(p) = m-\rank A(p) = \dim N$, that is, if and only if $\rank A(p) = \rank \tilde{A}(p)$.

    Observe that we always have $\rank A(p) \leq \rank \tilde{A}(p) \leq \rank A(p) + 1$. Let us suppose that $\rank \tilde{A}(p) \neq \rank A(p) + 1$, so that $\rank \tilde{A}(p) = \rank A(p)$ and $\wtilde{N} = N$. Now, if a vector field $\chi = \sum_{i=1}^m a_i \delta_i$ verifies $\chi(p) = 0$, then $\bar{a}(p) = (a_1(p), \ldots, a_m(p))^t$ belongs to $N = \wtilde{N}$, so $\sum_{i=1}^m a_i(p) \alpha_i(p) = 0$ and there cannot be a strong Euler vector field at $p$. 

    Reciprocally, if $\rank \tilde{A}(p) = \rank A(p) + 1$, then $\wtilde{N}$ is strictly contained in $N$. Therefore, there exists $\bar{a} = (a_1, \ldots, a_m)^t \in \C^m$ such that $\sum_{i=1}^m a_i \delta_i(p) = 0$ and $\sum_{i=1}^m a_i \alpha_i(p) \neq 0$. Taking $\chi = \sum_{i=1}^m a_i \delta_i$ we get the desired conclusion.
    
\end{proof}

As a consequence of this result, we get a criterion for strong Euler-homogeneity in terms of the analytic sets $D_i$ and $\tilde{D}_i$:

\begin{Th}[Geometric criterion for strong Euler-homogeneity]
\label{ThGeomEH}
    Let $D$ be a germ of divisor in $(\C^n,0)$ with $n \geq 2$. Then we have a decomposition of the singular locus $\Sing D = D_E \sqcup D_{\text{NE}}$, where $D_E = \bigsqcup_{i=0}^{n-2} (D_i \setminus \tilde{D}_i)$ is the set of singular points at which $D$ is strongly Euler-homogeneous and $D_{\text{NE}} = \bigsqcup_{i=0}^{n-2} (\tilde{D}_i \setminus D_{i-1})$ is the set of singular points at which $D$ is not strongly Euler-homogeneous. Then, $\tilde{D}_0 = D_0 \cap D_{\text{NE}}$ and

    \begin{itemize}
        \item[(1)] $D$ is strongly Euler-homogeneous on $D_0$ if and only if $\tilde{D}_0 = \varnothing$.
        \item[(2)] $D$ is strongly Euler-homogeneous on $D \setminus D_0$ if and only if $\tilde{D}_i = D_{i-1} \text{ for } i=1, \ldots, n-2$.
        \item[(3)] $D$ is strongly Euler-homogeneous if and only if $\tilde{D}_0 = \varnothing$ and $\tilde{D}_i = D_{i-1} \text{ for } i=1, \ldots, n-2$.

        \vspace{0,1 cm}
        
        In particular, for $n=2$, $D$ is strongly Euler-homogeneous if and only if $\tilde{D}_0 = \varnothing$.
    \end{itemize}
\end{Th}

\begin{proof}
    The inequality $\rank \tilde{A}(p) \leq \rank A(p) + 1$ shows that $D_i \setminus \tilde{D}_i$ is the set of points at which $\rank \tilde{A}(p) = \rank A(p) + 1 = i+1$. Similarly, the inequality $\rank A(p) \leq \rank \tilde{A}(p)$ shows that $\tilde{D}_i \setminus D_{i-1}$ is the set of points at which $\rank A(p) = \rank \tilde{A}(p) = i$. The first assertion then follows from the equality $\Sing D = D_{n-2}$ and Proposition \ref{PropRank}.

    The inclusions $D_0 \setminus \tilde{D}_0 \subset D_E$ and $\tilde{D}_0 \subset D_{\text{NE}}$ imply that $\tilde{D}_0 = D_0 \cap D_{\text{NE}}$, from which $(1)$ follows immediately. Now, the condition ``$D$ is strongly Euler-homogeneous on $D \setminus D_0$'' can be rewritten as $D_{\text{NE}} \subset D_0$, which is equivalent to $\tilde{D}_0 = D_{\text{NE}} = \tilde{D}_0 \sqcup \bigsqcup_{i=1}^{n-2} (\tilde{D}_i \setminus D_{i-1})$. Since $D_{i-1} \subset \tilde{D}_i$ for all $i$, this happens if and only if $\tilde{D}_i = D_{i-1}$ for all $i=1, \ldots, n-2$ and we get $(2)$. Finally, $(3)$ is straightforward from $(1)$ and $(2)$.
   
\end{proof}

\begin{Rmk}
\label{RmkProductInduction}
    One might ask why we should care about strong Euler-homogeneity only on $D_0$ or $D \setminus D_0$. Here is one reason:

Suppose that we want to prove that a divisor satisfying a certain property $P$ is always strongly Euler-homogeneous. Assume we know how to prove it for plane curves and that the property $P$ is preserved under smooth factors (as it happens with strong Euler-homogeneity). Then, we can apply an inductive argument: suppose the result is true for ambient dimension $k-1$ and consider a divisor $D$ in ambient dimension $k$ satisfying $P$. As $D$ is a product with a smooth factor at every point outside $D_0$, by induction hypothesis, we deduce that $D$ is strongly Euler-homogeneous on $D \setminus D_0$. Thus, we will only have to show that $P$ implies strong Euler-homogeneity on $D_0$ (or $\tilde{D}_0 = \varnothing$ by Theorem \ref{ThGeomEH}). And any equivalent or necessary condition for strong Euler-homogeneity outside $D_0$ could be useful for this purpose.

This strategy will be used in the proof of Theorems \ref{ThSHandSEH} and
\ref{ThConjDim4}.
\end{Rmk}

Hilbert's Nullstellensatz gives us an immediate algebraic translation of the developed criterion: 

\begin{Cor}[Algebraic criterion for strong Euler-homogeneity]
\label{CorAlgEH}
    Let $D$ be a germ of divisor in $(\C^n,0)$ with $n \geq 2$. Then:

    \begin{itemize}
        \item[(1)] $D$ is strongly Euler-homogeneous on $D_0$ if and only if $\tilde{I}_1 = \hol$.
        \item[(2)] $D$ is strongly Euler-homogeneous on $D \setminus D_0$ if and only if $\sqrt{\tilde{I}_{i+1}} = \sqrt{I_i}$ \text{ for } $i=1, \ldots, n-2$.
        \item[(3)] $D$ is strongly Euler-homogeneous if and only if $\tilde{I}_1 = \hol$ and $\sqrt{\tilde{I}_{i+1}} = \sqrt{I_i}$ \text{ for } $i=1, \ldots, n-2$.
    \end{itemize}
\end{Cor}

\begin{Rmk}
\label{RmkSEHFitting}

We have formulated this criterion in terms of ideals of minors of $A$ and $\tilde{A}$, which is more practical for our purpose. In terms of Fitting ideals, this criterion reads as:

\begin{itemize}
    \item $D$ is strongly Euler-homogeneous on $D_0$ if and only if $\sqrt{\Fitt_n(\mathcal{J}_f)} = \sqrt{\Fitt_n \left(\mathcal{J}_f/(f)\right)} \hspace{0,2 cm} (= \hol)$.

    \item $D$ is strongly Euler-homogeneous on $D \setminus D_0$ if and only if $\sqrt{\Fitt_i(\mathcal{J}_f)} = \sqrt{\Fitt_i \left(\mathcal{J}_f/(f)\right)}$ for all $i=2, \ldots, n-1$.

    \item $D$ is strongly Euler-homogeneous if and only if $\sqrt{\Fitt_i(\mathcal{J}_f)} = \sqrt{\Fitt_i \left(\mathcal{J}_f/(f)\right)}$ for all $i=2, \ldots, n$.
\end{itemize}
\end{Rmk}

We present now two examples in which we apply the developed criterion to determine if they are strongly Euler-homogeneous or not. Both of them are strongly Euler-homogeneous at the origin but the first one is also strongly Euler-homogeneous outside the origin and the second one is not. Calculations have been made with \textsc{Macaulay2} \cite{M2}:

\begin{Ex}
\label{ExSEH}

Let $D = V(f)$ in $(\C^4,0)$ with $f=xy(x+z)(x^2+yz)(z+yt)$. A minimal generating set of $\Der_f$ is:

\vspace{-0,2 cm}

$$\begin{array}{l}

    \delta_1 = x \partial_x + y \partial_y + z \partial_z, \\
    \delta_2 = (z+yt) \partial_t, \\
    \delta_3 = y(x+z) \partial_y -z(x+z) \partial_z -2t(x+z) \partial_t, \\
    \delta_4 = y^2(y+z) \partial_y + y(x^2-z^2) \partial_z + (-x^2+yz+2z^2) \partial_t, \\
    \delta_5 = yz(y+z) \partial_y + z(x^2-z^2) \partial_z + (x^2t-2z^2t+z^2) \partial_t. \\

\end{array}$$

As $\delta_1(f) = 6f$, $\delta_2(f)=yf$, $\delta_3(f) = -zf$, $\delta_4(f)=(xy+3y^2+yz)f$ and $\delta_5(f) = (x^2+xz+3yz-z^2)f$, the extended Saito matrix is:

$$\tilde{A} = \left( \begin{array}{cccc|c}
    x & y & z & 0 & -6 \\
    0 & 0 & 0 & z+yt & -y \\
    0 & y(x+z) & -z(x+z) & -2t(x+z) & z \\
    0 & y^2(y+z) & y(x^2-z^2) & -x^2+yz+2z^2 & -xy-3y^2-yz \\
    0 & yz(y+z) & z(x^2-z^2) & x^2t-2z^2t+z^2 & -x^2-xz-3yz+z^2\\
\end{array} \right).$$

By Corollary \ref{CorAlgEH}, $D$ is strongly Euler-homogeneous if and only if $\tilde{I}_1 = \hol$ (this holds because the last column of $\tilde{A}$ contains a unit) and $\sqrt{\tilde{I}_{i+1}} = \sqrt{I_i}$ for $i=1,2$. We have: 

$$\begin{array}{l}

    \sqrt{\tilde{I}_2} = \sqrt{I_1} = (x,y,z), \\[0,2 cm]

    \sqrt{\tilde{I}_3} = \sqrt{I_2} = (x, z, yt). \\[0,2 cm]
\end{array}$$

We conclude that $D$ is strongly Euler-homogeneous.
\end{Ex}

\begin{Ex}
\label{ExNotSEH}
Let $D = V(f)$ in $(\C^4,0)$ with $f=xy(x+z)(x^2+yz)(z^2+yt)$ (note that the only difference with the previous example is that the $z$ in the last term of the product is squared). A minimal generating set of $\Der_f$ is:

\vspace{-0,2 cm}

$$\begin{array}{l}

    \delta_1 = x \partial_x + y \partial_y + z \partial_z + t \partial_t, \\
    \delta_2 = (z^2+yt) \partial_t, \\
    \delta_3 = y(x+z) \partial_y -z(x+z) \partial_z -3t(x+z) \partial_t, \\
    \delta_4 = y^2(y+z) \partial_y + y(x^2-z^2) \partial_z - (2x^2z+y^2t+3yzt) \partial_t, \\
    \delta_5 = yz(y+z) \partial_y + z(x^2-z^2) \partial_z + (2x^2t-yzt+3yt^2) \partial_t. \\

\end{array}$$

We have $\delta_1(f) = 7f$, $\delta_2(f)=yf$, $\delta_3(f) = -(x+2z)f$, $\delta_4(f)=(xy+2y^2-2yz)f$ and $\delta_5(f) = (2x^2+xz+2yz-2z^2+3yt)f$. The extended Saito matrix is then:

\begin{equation*}
\resizebox{.99\hsize}{!}{
    $\tilde{A} = \left( \begin{array}{cccc|c}
    x & y & z & t & -7 \\
    0 & 0 & 0 & z^2+yt & -y \\
    0 & y(x+z) & -z(x+z) & -3t(x+z) & x+2z \\
    0 & y^2(y+z) & y(x^2-z^2) & -(2x^2z+y^2t+3yzt) & -xy-2y^2+2yz \\
    0 & yz(y+z) & z(x^2-z^2) & 2x^2t-yzt+3yt^2 & -2x^2-xz-2yz+2z^2-3yt\\
\end{array} \right)$.}\end{equation*}

As before, $D$ is strongly Euler-homogeneous if and only if $\tilde{I}_1 = \hol$ (which holds because of the $-7$ in the last column) and $\sqrt{\tilde{I}_{i+1}} = \sqrt{I_i}$ for $i=1,2$. But we have: 

$$\begin{array}{l}
    \sqrt{\tilde{I}_2} = (x,y,z) \subsetneq \sqrt{I_1} = (x,y,z,t), \\[0,2 cm]

    \sqrt{\tilde{I}_3} = \sqrt{I_2} = (t(y+z), t(x+z), z(z-t), z(y+t), z(x+t), xy - zt). \\[0,2 cm]
\end{array}$$

This means that $D_0 \subsetneq \tilde{D}_1$ and $D$ is not strongly Euler-homogeneous at any point of $\tilde{D}_1 \setminus D_0 = \{(0,0,0,t) \mid t \neq 0\}$, where both matrices $A$ and $\tilde{A}$ have rank $1$. \\
\end{Ex}

The following lemma states that the vanishing locus of $\partial f : f$ is always inside the singular locus. This will give us a necessary condition for a divisor to be strongly Euler-homogeneous outside $D_0$. 

\begin{Lem}
\label{LemAlphaSubsetD}
    $V(\partial f : f) \subset \Sing D$.
\end{Lem}

\begin{proof}
    Recall that, with the previous notation, $\partial f : f = (\alpha_1, \ldots, \alpha_m)$. As $f$ is integral over the ideal generated by its partial derivatives \cite[Satz 5.2]{Scheja70}, there exists a logarithmic derivation $\eta$ and an integer $r \geq 1$ such that $\eta(f) = f^r$. Let us write $\eta = \sum_{i=1}^m c_i \delta_i$, so that we get $f^{r-1} = \sum_{i=1}^m c_i \alpha_i$ when we apply $\eta$ to $f$. If $r=1$, then $V(\alpha_1, \ldots, \alpha_m) = \varnothing$ and there is nothing to prove. Otherwise, it is clear that $V(\alpha_1, \ldots, \alpha_m) \subset D$. As $D = D_{n-1}$, if $p \in V(\alpha_1, \ldots, \alpha_m)$, then $\rank \tilde{A}(p) = \rank A(p) \leq n-1$, so $p \in \tilde{D}_{n-1} = \Sing D$.
    
\end{proof}

\begin{Cor}[Necessary condition for strong Euler-homogeneity outside $D_0$]
\label{CorAlphaDtilde0}
    If $D = V(f)$ is a germ of divisor in $(\C^n,0)$ that is strongly Euler-homogeneous on $D \setminus D_0$, then $\tilde{D}_0 = V(\partial f : f)$ and $\sqrt{\tilde{I}_1} = \sqrt{\partial f : f}$.
\end{Cor}

\begin{proof}
      It is clear that $\tilde{D}_0 \subset V(\alpha_1, \ldots, \alpha_m) = V(\partial f : f)$. For the converse, consider a point $p \in V(\alpha_1, \ldots, \alpha_m)$. In particular, by Lemma \ref{LemAlphaSubsetD}, $p \in \Sing D \subset D$. As the last column of $\tilde{A}$ vanishes at $p$, we have that $\tilde{A}(p)$ and $A(p)$ have the same rank. Thus, by Proposition \ref{PropRank}, $D$ is not strongly Euler-homogeneous at $p$. But then $p$ must belong to $D_0$ by hypothesis, so $0 = \rank A(p) = \rank \tilde{A}(p)$ and $p \in \tilde{D}_0$. The equality $\sqrt{\tilde{I}_1} = \sqrt{\partial f : f}$ follows from the first one by Hilbert's Nullstellensatz.
\end{proof}

This result tells us that, when $D$ is strongly Euler-homogeneous outside $D_0$, the radical of $\partial f : f$ does not depend on the choice of $f$. The following example shows that, if strong Euler-homogeneity outside $D_0$ is not assumed, this may no longer be true:

\begin{Ex}
    Let $D=V(f)$ in $(\C^3,0)$ with $f = xyz(x^3+xyz+y^3z^3)$. This divisor is not strongly Euler-homogeneous outside $D_0$, as $\sqrt{I_1} = (x,y,z) \neq (x,yz) = \sqrt{\tilde{I_2}}$. And, if we set $g=(1+y)f$, we have $\sqrt{\partial g : g} = (x,y) \neq (x, yz) = \sqrt{\partial f : f}$. \\

\end{Ex}

\section{Saito-holonomic type properties}
\label{SectionSaitoHolonomic}

Let us begin this section by recalling the different notions of Koszul-freeness defined in \cite[Definitions 1.10 and 1.13]{NarvDual} and \cite[Remark 7.3]{SymmetryGS}. We consider the graded ring $\gr \D_{X,p}$ associated with the order filtration. For a differential operator $P \in \D_{X,p}$, we will denote by $\sigma(P)$ the principal symbol of $P$ in $\gr \D_{X,p}$. It is well-known that $\gr \D_{X,p}$ can be identified with the commutative ring $\hol_{X,p}[\underline{\xi}]$, where $\underline{\xi} = (\xi_1, \ldots, \xi_n)$ and $\xi_i = \sigma(\partial_i)$.

\begin{Def}
\label{DefKoszul}
    Let $D \subset X$ be a free divisor. We say $D$ is:

    \begin{itemize}
        \item \emph{weakly Koszul-free} at $p \in D$ if, for some (or any) basis $\{\delta_1, \ldots, \delta_n\}$ of $\Der_{X,p}(-\log D)$ and some (or any) reduced local equation $f$ of $D$ at $p$, the sequence

        $$ \sigma(\delta_1)-\alpha_1 s, \ldots, \sigma(\delta_n)-\alpha_n s, \text{ with } \alpha_i \in \hol_{X,p} \text{ being such that } \delta_i(f) = \alpha_i f,$$

        \vspace{0,2 cm}

        \noindent is regular in $\gr \D_{X,p}[s]$.

        \item \emph{Koszul-free} at $p \in D$ if, for some (or any) basis $\{\delta_1, \ldots, \delta_n\}$ of $\Der_{X,p}(-\log D)$, the sequence 
        
        $$\sigma(\delta_1), \ldots, \sigma(\delta_n)$$
        
        is regular in $\gr \D_{X,p}$.

        \item \emph{strongly Koszul-free} at $p \in D$ if, for some (or any) basis $\{\delta_1, \ldots, \delta_n\}$ of $\Der_{X,p}(-\log D)$ and some (or any) reduced local equation $f$ of $D$ at $p$, the sequence

        $$ f, \sigma(\delta_1)-\alpha_1 s, \ldots, \sigma(\delta_n)-\alpha_n s, \text{ with } \alpha_i \in \hol_{X,p} \text{ being such that } \delta_i(f) = \alpha_i f, $$

        \vspace{0,2 cm}

        \noindent is regular in $\gr \D_{X,p}[s]$.
        
    \end{itemize}

    We say $D$ is \emph{weakly Koszul-free}, \emph{Koszul-free} or \emph{strongly Koszul-free} if it is so at each $p \in D$.
\end{Def}

Our purpose is to define and study analogous properties in the general case in which our divisor is not necessarily free. In order to do this, we need first to make some algebraic considerations:

In a Cohen-Macaulay local ring, a sequence of $r$ elements is regular if and only if the height of the ideal they generate is $r$ \cite[Theorem 17.4]{MatsumuraGris}. The ring $\hol_{X,p}[\underline{\xi}] = \C\{\underline{x}\}[\underline{\xi}]$ (resp. $\hol_{X,p}[\underline{\xi},s]$) is not local but we will see that, for homogeneous elements with respect to $\underline{\xi}$ (resp. with respect to $\underline{\xi}, s$), this is still true. In order to do so, we will study the extension $\C\{\underline{x}\}[\underline{\xi}] \subset \C\{\underline{x},\underline{\xi}\}$. 

If a ring map $A \to B$ is faithfully flat, then a sequence of elements is regular in $A$ if and only if the corresponding sequence is regular in $B$. Moreover, if $A$ and $B$ are Noetherian, then the height of an ideal $I \subset A$ (denoted by $\htt I$) coincides with the height of its extension $B \cdot I$ in $B$ \cite[(4.C), (5.D) and (13.B.3)]{MatsumuraRosa}. The extension $\C\{\underline{x}\}[\underline{\xi}] \subset \C\{\underline{x},\underline{\xi}\}$ is not faithfully flat but, for homogeneous elements and ideals, these statements still hold. For the sake of completeness, we include the proof here: 

\begin{Prop}
\label{PropHeight}

    Let $a_1, \ldots, a_r \in \C\{\underline{x}\}[\underline{\xi}]$ be homogeneous elements with respect to $\underline{\xi}$ and consider the ring extension $\C\{\underline{x}\}[\underline{\xi}] \subset \C\{\underline{x},\underline{\xi}\}$. Then:

    \begin{itemize}
        \item[(1)] $a_1, \ldots, a_r$ is a regular sequence in $\C\{\underline{x}\}[\underline{\xi}]$ if and only if it is so in $\C\{\underline{x},\underline{\xi}\}$.
        \item[(2)] If $I = (a_1, \ldots, a_r)$, then $\htt I = \htt I^e$, where $I^e$ denotes the extended ideal.
    \end{itemize}

    Consequently, $a_1, \ldots, a_r$ is a regular sequence in $\C\{\underline{x}\}[\underline{\xi}]$ if and only if $\htt(a_1, \ldots, a_r) = r$.
\end{Prop}

\begin{proof}
    Let $A = \C\{\underline{x}\}[\underline{\xi}]$. As the extension $\C\{\underline{x},\underline{\xi}\} \subset \C[[\underline{x},\underline{\xi}]]$ is faithfully flat, we just need to prove the statements for the extension $\C\{\underline{x}\}[\underline{\xi}] \subset \C[[\underline{x},\underline{\xi}]]$. Let $M=(\underline{x},\underline{\xi})$ be the maximal homogeneous ideal of $A$, so that the $M$-adic completion of $A$ is $\hat{A} = \C[[\underline{x},\underline{\xi}]]$, and let $S = 1+M$. By \cite[(24.A) and (24.B)]{MatsumuraRosa}, the canonical map $A \to \hat{A}$ factors through $S^{-1}A$ and the extension $S^{-1}A \to \hat{A}$ is faithfully flat. Thus, it remains to show the results for the extension $A \to S^{-1} A$:

    It is easy to see that, if $J$ is a homogeneous ideal, $a \in A$, $s \in S$ and $as \in J$, then $a \in J$. From this, we obtain that $(S^{-1} J)^c = J$, where the superindex $c$ means the contraction of the ideal. This allows us to deduce that $a_1, \ldots, a_r$ is a regular sequence in $\C\{\underline{x}\}[\underline{\xi}]$ if and only if so is $a_1/1, \ldots, a_r/1$ in $S^{-1}A$, which gives $(1)$. 

    Let us prove $(2)$. Consider a homogeneous (with respect to $\underline{\xi}$) prime ideal $\mathfrak{p}$. Note that $S$ cannot meet $\mathfrak{p}$ because the $0$-th degree part of an element in the intersection would be a unit belonging to $\mathfrak{p}$. Consequently, we have $\htt \mathfrak{p} = \htt (S^{-1} \mathfrak{p})$ \cite[AC VIII.9, Corollaire after Proposition 7]{Bourbaki89}. As minimal primes over homogeneous ideals are homogeneous \cite[AC VIII.64, Lemme 1.c)]{Bourbaki89}, we deduce that $\htt I = \htt (S^{-1}I)$, which finishes the proof.

    The last statement follows from the previous ones and the fact that $\C\{\underline{x},\underline{\xi}\}$ is a Cohen-Macaulay local ring.
    
\end{proof}

Now, we are ready to generalize Koszul type properties:

\begin{Def}
\label{DefSaitoHolonomic}
    Let $D \subset X$ be a divisor. We say:

    \begin{itemize}
        \item $D$ is \emph{weakly Saito-holonomic} at $p \in D$ if, for some (or any) generating set $\mathcal{S} = \{\delta_1, \ldots, \delta_m\}$ of $\Der_{X,p}(-\log D)$ and some (or any) reduced local equation $f$ of $D$ at $p$, 
        
        $$\htt(\sigma(\delta_1)-\alpha_1 s, \ldots, \sigma(\delta_m)-\alpha_m s) = n \text{ in } \gr \D_{X,p}[s],$$
        
        where $\alpha_i \in \hol_{X,p}$ is such that $\delta_i(f) = \alpha_i f.$

        \item $D$ is \emph{Saito-holonomic} at $p \in D$ if, for some (or any) generating set $\mathcal{S} = \{\delta_1, \ldots, \delta_m\}$ of $\Der_{X,p}(-\log D)$, 
        
        $$\htt(\sigma(\delta_1), \ldots, \sigma(\delta_m)) = n  \text{ in }\gr \D_{X,p}.$$

        \item $D$ is \emph{strongly Saito-holonomic} at $p \in D$ if, for some (or any) generating set $\mathcal{S} = \{\delta_1, \ldots, \delta_m\}$ of $\Der_{X,p}(-\log D)$ and some (or any) reduced local equation $f$ of $D$ at $p$, 
        
        $$\htt(f,\sigma(\delta_1)-\alpha_1 s, \ldots, \sigma(\delta_m)-\alpha_m s) = n+1 \text{ in } \gr \D_{X,p}[s],$$
        
        where $\alpha_i \in \hol_{X,p}$ is such that $\delta_i(f) = \alpha_i f.$
        
    \end{itemize}
We say $D$ is \emph{weakly Saito-holonomic}, \emph{Saito-holonomic} or \emph{strongly Saito-holonomic} if it is so at each $p \in D$.
\end{Def}

\begin{Rmk}
$ $
\begin{itemize}
    \item[(a)] For free divisors, these definitions coincide with those of Koszul-freeness given in Definition \ref{DefKoszul} by virtue of Proposition \ref{PropHeight}.
    
    \item[(b)] In \cite[Theorem 7.4]{GMNS}, the authors state that holonomic free divisors in the sense of K. Saito (see \cite[Definition 3.8]{Saito80}) are exactly Koszul-free divisors. The same argument is valid in the general case to conclude that our definition of holonomicity is equivalent to the one given by K. Saito. This explains the name we have chosen for this property, in agreement with \cite[Definition 2.5]{Uli17}.
    
\end{itemize}
\end{Rmk}

By coherence, if $D$ satisfies any of the properties in Definition \ref{DefSaitoHolonomic} at $p$, then it does so in a neighbourhood of $p$. This will allow us to characterize divisors in $(\C^n,0)$ satisfying them in a neighbourhood of $0$ just by looking at the dimensions of the germs $D_i$ and $\tilde{D}_i$ defined in the previous section. In order to do this, we need first a technical lemma:

\begin{Lem}
\label{LemForSH}
    Let $E$ be a non-empty analytic germ in $(\C^n,0)$ given by $k$ equations with $0 \leq k \leq n$ and let $B(x)$ be an $m \times r$ matrix with entries in $\hol$. Let $W=\{(x,y) \in (E, 0) \times (\C^r, 0) \mid B(x)y=0\}$ and let $B_i = \{p \in (E,0) \mid \rank B(p) \leq i\}$. Then, $\dim W \leq r-k$ if and only if $B_i = \varnothing$ for all $i < k$ and $\dim B_i \leq i-k$ for all $k \leq i \leq s$, where $s=\min\{m,r\}$. In particular, as $B_s = E$, this can only hold when $s \geq n$.
\end{Lem}
\begin{proof}
    Every germ of analytic set has a representative with the same dimension. In this proof we need to work with these representatives. To that end, we choose small enough neighbourhoods $U$ of $0 \in \C^n$ and $V$ of $0 \in \C^r$ and representatives in them of all the sets we are considering so that

    \begin{itemize}
        
        \item their dimensions agree with those of their respective germs,

        \item the entries of the matrix $B$ are defined in $U$, and

        \item the representative of $W$ has a finite decomposition into irreducible components (this can always be done by \cite[Ch. IV, \S3.1, Proposition 3.a]{Lojasiewicz}).
    \end{itemize}
    
    \noindent For simplicity, these sets will be denoted by the same symbol as the germ they represent. 

    Let us consider the projection $\pi: W \to E$. Let $0 \leq i \leq s$ be such that $B_i \neq \varnothing$, so that $\dim B_i \geq 0$. As $(p,0) \in W$ for every $p \in E$, the map $\pi:W \to E$ is surjective, and so is the restriction $\pi: \pi^{-1}(B_i) \to B_i$. 
    
    Note that $\pi^{-1}(p) = \{p\} \times \{y \in V \mid B(p) y = 0\}$. As the second factor is a piece of linear space, its dimension is constant at each point and equal to $r-\rank B(p)$. Thus, for every $p \in B_i$ we deduce that $\dim \pi^{-1}(p) \geq r-i$ ($\geq 0$ since $r \geq s$) and we can apply \cite[Ch. V, \S3.2, Theorem 2]{Lojasiewicz} to deduce that $\dim \pi^{-1}(B_i) \geq \dim \pi(\pi^{-1}(B_i)) + r-i = \dim B_i + r-i$.
    
    If $\dim W \leq r-k$, then

    \vspace{-0,1 cm}

    $$r-k \geq \dim W \geq \dim \pi^{-1}(B_i) \geq \dim B_i + r-i \hspace{0,2 cm} (\geq r-i).$$

    \vspace{0,2 cm}

    In particular, $i \geq k$ (so $B_i = \varnothing$ for $i < k$) and $\dim B_i \leq i-k$. 
    
    Note that $B_s = E$. Since $E$ is given by $k$ equations, $\dim E \geq n-k$ and so $n-k \leq \dim E = \dim B_s \leq s-k$. Therefore, $s \geq n$.

    To prove the converse, consider an irreducible component $W'$ of $W$. Let us distinguish two cases:
    
    \begin{itemize}
        \item If $W' \subset \pi^{-1}(B_0)$, in particular, $B_0 \neq \varnothing$, so $k=0$ and $\dim B_0 = 0$ by hypothesis. As $p \in B_0$ implies $\rank B(p) = 0$, we have $\dim \pi^{-1}(p) = r$ for every $p \in B_0$. Then, $\dim \pi^{-1}(B_0) = \dim B_0 + r = r$ and $\dim W' \leq \dim \pi^{-1}(B_0) = r = r-k$.  
        
        \item If $W' \not\subset \pi^{-1}(B_0)$, as $W' \subset W = \pi^{-1}(E) = \pi^{-1}(B_s)$, there exists $1 \leq i \leq s$ for which $W'$ is contained in $\pi^{-1}(B_i)$ but not in $\pi^{-1}(B_{i-1})$. In particular, $B_i \neq \varnothing$ and so $i \geq k$ by hypothesis.

        Let $\pi' := \pi|_{W'}: W' \to E$ and let $\lambda(\pi')$ be the generic dimension of the fibres of $\pi'$, which is defined as

        \vspace{-0,3 cm}

        $$\lambda(\pi'):= \min \{\dim (\pi'^{-1}(\pi'(x,y)))_{(x,y)} \mid (x,y) \in W'\}.$$ 

        Since $W'$ is irreducible, we have the equality $\dim W' = \lambda(\pi') + \dim \pi'(W')$ \cite[Ch. V, \S3.3, formula (1) and \S3.2, Theorem 4]{Lojasiewicz}. Let us bound each of these summands:

        On one hand, for each $(x,y) \in W'$, we have
        
        $$\pi'^{-1}(\pi'(x,y)) = \pi'^{-1}(x) = \pi^{-1}(x) \cap W' \subset \pi^{-1}(x), $$
        
        so $\lambda(\pi') \leq \dim (\pi^{-1}(x))_{(x,y)}$ for every $(x,y) \in W'$. 
        
        Since $W'$ is not contained in $\pi^{-1}(B_{i-1})$, there exists $(p,q) \in W' \setminus \pi^{-1}(B_{i-1}) \subset \pi^{-1}(B_i) \setminus \pi^{-1}(B_{i-1})$. Then, $p = \pi(p,q) \in B_i \setminus B_{i-1}$ and so $\rank B(p) = i$. As explained before, $\pi^{-1}(p)$ is then of constant dimension $r-\rank B(p) = r-i$. We deduce that

        \vspace{-0,2 cm}
        
        $$\lambda(\pi') \leq \dim (\pi^{-1}(p))_{(p,q)} = r-i. $$

        On the other hand, $\pi' (W') = \pi(W') \subset B_i$, so $\dim \pi'(W') \leq \dim B_i$. 
        
        We finally get
        
        $$ \dim W' = \lambda(\pi') + \dim \pi'(W') \leq r-i+\dim B_i \leq r-k,$$
        
        where the last inequality holds by hypothesis.

    \end{itemize}

    We conclude that every irreducible component of $W'$ verifies $\dim W' \leq r-k$ and, thus, $\dim W \leq r-k$.
    
\end{proof}

In what follows, we will need to consider the following map of graded algebras:

$$\begin{array}{cccccc}
    \varphi: & \hol[\underline{\xi},s] & \longrightarrow & \Rees(\mathcal{J}_f) \\
    & \xi_i & \mapsto & \partial_i(f) \cdot t \\
    & s & \mapsto & f \cdot t
\end{array}$$

\noindent where $\Rees(\mathcal{J}_f) =\bigoplus_{i=0}^\infty \mathcal{J}_f^i t^i $ is the Rees algebra of $\mathcal{J}_f$. As $\varphi$ is surjective and $\Rees(\mathcal{J}_f)$ is an integral domain, $\ker \varphi$ is a prime ideal. Since the dimension of $\hol[\underline{\xi},s]$ is $2n+1$ and that of $\Rees(\mathcal{J}_f)$ is $n+1$ (cf. \cite[Theorem 5.1.4]{ReesDim}), we deduce that $\htt(\ker \varphi) = n$.

The homogeneous part of degree $1$ of $\ker \varphi$ is generated as an $\hol$-module by the set $\{\sigma(\delta_i)-\alpha_i s, i=1, \ldots, m\}$ (in the notation of Definition \ref{DefSaitoHolonomic}).  The ideal of $\hol[\underline{\xi},s]$ generated by these elements is denoted by $\ker^{(1)} \varphi$. A divisor $D$ is said to be of \emph{linear Jacobian type} if $\ker \varphi = \ker^{(1)} \varphi$.

\begin{Th}
\label{ThCaractSH}
    Let $D = V(f)$ be a germ of divisor in $(\C^n,0)$. Then:

    \begin{itemize}
        \item[(1)] $D$ is \emph{weakly Saito-holonomic} if and only if $\dim \tilde{D}_i \leq i$ for all $i=0, \ldots, n-3$.

        \item[(2)] $D$ is \emph{Saito-holonomic} if and only if $\dim D_i \leq i$ for all $i=0, \ldots, n-3$.

        \item[(3)] $D$ is \emph{strongly Saito-holonomic} if and only if $\tilde{D}_0 = \varnothing$ and $\dim \tilde{D}_i \leq i-1$ for all $i=1, \ldots, n-2$.
        
    \end{itemize}
    
\end{Th}

\begin{proof}

Let $\mathcal{S} = \{\delta_1, \ldots, \delta_m\}$ be a generating set of $\Der_f$ and let $A$ be the Saito matrix with respect to $\mathcal{S}$. Let $\alpha_i \in \hol$ such that $\delta_i(f)=\alpha_i f$ for all $i=1, \ldots, m$ and let $\tilde{A}$ be the extended Saito matrix. By coherence, it is enough to show these properties hold at $0$:

    \begin{itemize}
        \item[(1)] $D$ is weakly Saito-holonomic at $0$ if and only if $\htt (\sigma(\delta_1)-\alpha_1 s, \ldots, \sigma(\delta_m)-\alpha_m s) = n$ in $\hol[\underline{\xi},s]$. As $(\sigma(\delta_1)-\alpha_1 s, \ldots, \sigma(\delta_m)-\alpha_m s)$ is a homogeneous ideal (with respect to $\xi_1, \ldots, \xi_n,s$), by Proposition \ref{PropHeight}, this is equivalent to $\htt (\sigma(\delta_1)-\alpha_1 s, \ldots, \sigma(\delta_m)-\alpha_m s) = n$ in $\C\{\underline{x},\underline{\xi},s\}$ or even $\dim W = n+1$, where $W = V(\sigma(\delta_1)-\alpha_1 s, \ldots, \sigma(\delta_m)-\alpha_m s) \subset (\C^{2n+1},0)$. As $(\sigma(\delta_1)-\alpha_1 s, \ldots, \sigma(\delta_m)-\alpha_m s) = \ker^{(1)} \varphi \subset \ker \varphi$ and $\ker \varphi$ is a prime ideal of $\hol[\underline{\xi},s]$ of height $n$, the inequality $\htt (\sigma(\delta_1)-\alpha_1 s, \ldots, \sigma(\delta_m)-\alpha_m s) \leq n$ (or $\dim W \geq n+1$) always holds. Thus, $D$ is weakly Saito-holonomic at $0$ if and only if $\dim W \leq n+1$. Note that we can write

        \vspace{-0,1 cm}
        
        $$W=\{(x,(\xi,s)) \in (\C^n,0) \times (\C^{n+1},0) \mid \tilde{A}(x)(\xi,s)^t=0\}.$$

        Now, we just need to apply Lemma \ref{LemForSH}. With the notation of the lemma, $E=\C^n, k=0$, $B(x)=\tilde{A}(x)$, $r=n+1$, $s = \min\{m,n+1\} \in \{n,n+1\}$ (recall that $m \geq n$) and $B_i = \tilde{D}_i$. Thus, $\dim W \leq n+1$ if and only if $\dim \tilde{D}_i \leq i$ for all $i=0, \ldots, s$. The conditions $\dim \tilde{D}_n \leq n$ and $\dim \tilde{D}_{n+1} \leq n+1$ (if needed) always holds because $\tilde{D}_n \subset \tilde{D}_{n+1} \subset \C^n$. Moreover, since $\tilde{D}_{n-2} \subset \tilde{D}_{n-1} = \Sing D$ and the singular locus of $D$ is of dimension at most $n-2$, we have $\dim \tilde{D}_{n-2} \leq \dim \tilde{D}_{n-1} \leq n-2 < n-1$. We conclude that $D$ is weakly Saito-holonomic if and only if $\dim \tilde{D}_i \leq i$ for all $i=0, \ldots, n-3$.

        \item[(2)] $D$ is Saito-holonomic at $0$ if and only if $\htt (\sigma(\delta_1), \ldots, \sigma(\delta_m)) = n$ in $\hol[\underline{\xi}]$, or in $\C\{\underline{x}, \underline{\xi}\}$ by Proposition \ref{PropHeight}. Since $(\sigma(\delta_1), \ldots, \sigma(\delta_m)) \subset (\xi_1, \ldots, \xi_n)$, we get $\htt (\sigma(\delta_1), \ldots, \sigma(\delta_m)) \leq n$. Thus, we just need to consider the other inequality, which is equivalent to $\dim W \leq n$, where $W = V(\sigma(\delta_1), \ldots, \sigma(\delta_m)) = \{(x,\xi) \in (\C^n,0) \times (\C^n,0) \mid 
        A(x)\xi^t=0\}$. We proceed as in $(1)$ with $E=\C^n, k=0, B(x)=A(x)$, $r=n$, $s=\min\{m,r\} = n$ and $B_i = D_i$ to obtain that $\dim W \leq n$ if and only if $\dim D_i \leq i$ for all $0 \leq i \leq n$. Since $D_n = \C^n$, $D_{n-1} = D$ and $D_{n-2} = \Sing D$, the last three conditions are automatically satisfied. 

        \item[(3)] $D$ is strongly Saito-holonomic at $0$ if and only if $\htt (f, \sigma(\delta_1)-\alpha_1 s, \ldots, \sigma(\delta_m)-\alpha_m s) = n+1$ in $\hol[\underline{\xi},s]$ or, reasoning as before, if $\dim W = n$, where $W = V(f,\sigma(\delta_1)-\alpha_1 s, \ldots, \sigma(\delta_m)-\alpha_m s) = \{(x,(\xi,s)) \in D \times (\C^{n+1},0) \mid 
        \tilde{A}(x)(\xi,s)^t=0\}$. Note that 

        $$\dim W \geq \dim V(\sigma(\delta_1)-\alpha_1 s, \ldots, \sigma(\delta_m)-\alpha_m s) - 1 \geq n, $$
        
        where the last inequality was argued in $(1)$. Thus, $D$ is strongly Saito-holonomic if and only if $\dim W \leq n$, where $W = V(f,\sigma(\delta_1)-\alpha_1 s, \ldots, \sigma(\delta_m)-\alpha_m s) = \{(x,(\xi,s)) \in D \times (\C^{n+1},0) \mid 
        \tilde{A}(x)(\xi,s)^t=0\}$. We apply again Lemma \ref{LemForSH} with $E=D, k=1, B(x)=\tilde{A}(x)$, $r=n+1$ and $s=\min\{m,n+1\} \in \{n,n+1\}$ to obtain that $\dim W \leq n$ if and only if $B_0 = \varnothing$ and $\dim B_i \leq i-1$ for all $i=1, \ldots, s$.  Now, note that
        $B_i = \{p \in D \mid \rank \tilde{A}(p) \leq i\} = \tilde{D}_i$ for $i \leq n-1$, since $\tilde{D}_i \subset D$ for $i \leq n-1$ (observe that in $\tilde{D}_i$ we allow $p$ to be in principle any point of $(\C^n,0)$, while in $B_i$ it must be a point of $D$). Since $B_{n-1} = \tilde{D}_{n-1} = \Sing D$ and $B_n \subset B_{n+1} = D$,  the condition $\dim B_i \leq i-1$ is always satisfied for $i \in \{n-1,n,n+1\}$. We conclude that $D$ is strongly Saito-holonomic if and only if $\tilde{D}_0 = \varnothing$ and $\dim \tilde{D}_i \leq i-1$ for $i=1, \ldots, n-2$.
    \end{itemize}
    
\end{proof}

\begin{Rmk}
$ $
\begin{itemize}
    \item[(a)] As $\tilde{D}_i \subset D_i \subset \tilde{D}_{i+1}$, it is now clear that strongly Saito-holonomic $\Rightarrow$ Saito-holonomic $\Rightarrow$ weakly Saito-holonomic.

    \item[(b)] If $D$ is Saito-holonomic or weakly Saito-holonomic in a punctured neighbourhood of a point $p$, then it is so at $p$. This is because outside the fibre at $p$, the set $W$ (in the notation of the proof of Theorem \ref{ThCaractSH}) has the required dimension ($n$ or $n+1$, respectively) and the fibre at $p$ (of maximal dimension $n$ or $n+1$, respectively) cannot increase that dimension.

    This argument cannot be applied to strong Saito-holonomicity, since we need $W$ to be $n$-dimensional and the fibre at $p$ could be of dimension $n+1$. In fact, any non-quasihomogeneous germ of plane curve (which is a non-strongly Euler-homogeneous isolated singularity) is a counterexample (see Theorem \ref{ThSHandSEH}).
    
    \item[(c)] The characterization given by K. Saito for holonomic divisors in terms of the dimension of the sets $D_i$ ($A_r$ in his notation) coincides with the one given in Theorem \ref{ThCaractSH} for Saito-holonomic divisors (see \cite[Definition 3.12 and Lemma 3.13]{Saito80}). This confirms the equivalence of the two definitions.
\end{itemize}
    
\end{Rmk}

Some straightforward consequences of Theorem \ref{ThCaractSH} in low dimension are:

\begin{Cor}
\label{CorSHLowdim}
Let $D$ be a germ of divisor in $(\C^n,0)$.

\begin{itemize}
    \item[(1)] If $n=2$, then $D$ is Saito-holonomic (this was first proved in \cite[Corollary 4.2.2]{Cald1999}).
    
    \item[(2)] If $n=3$ and $D$ is strongly Euler-homogeneous, then it is weakly Saito-holonomic. 
\end{itemize}
\end{Cor}

\begin{proof}
    For $n=2$, the conditions for Saito-holonomicity in Theorem \ref{ThCaractSH} are automatically satisfied. 
    For $n=3$, $D$ is weakly Saito-holonomic if and only if $\dim \tilde{D}_0 \leq 0$. But strong Euler-homogeneity implies $\tilde{D}_0 = \varnothing$, so this holds.
    
\end{proof}

As it happens with strong Euler-homogeneity and Koszul-freeness \cite[Proposition 1.10]{CN02}, the three versions of Saito-holonomicity are well-behaved with respect to smooth factors: 

\begin{Prop}
\label{PropProductSH}
    Let $D'$ be a germ of divisor in $(\C^{n-1},0)$ and let $D = D' \times \C \subset (\C^n,0)$. Then $D$ is weakly Saito-holonomic, Saito-holonomic or strongly Saito-holonomic if and only if so is $D'$.
\end{Prop}

\begin{proof}
    Let $D' = V(g)$, where $g \in \C\{x_2, \ldots, x_n\}$ and let $f(x_1, \ldots, x_n) = g(x_2, \dots, x_n)$, so that $D = V(f)$. If $\mathcal{S}' = \{\delta_1, \ldots, \delta_m\}$ is a generating set for $\Der_g$, then $\mathcal{S} = \{\partial_1, \delta_1, \ldots, \delta_m\}$ is a generating set for $\Der_f$ (where $\partial_1(f) = 0$). Thus, if $A'$ and $\tilde{A}'$ are the Saito matrices for $D'$ with respect to $\mathcal{S}'$ (and $g$), then the Saito matrices for $D$ with respect to $\mathcal{S}$ (and $f$) are:

$$ A = \left(\begin{array}{@{}c|c@{}}
  1 & \begin{matrix}
  0 & \ldots & 0 \\
  \end{matrix} \\
\hline
  \begin{matrix}
  0 \\
  \vdots \\
  0
  \end{matrix} & A'
\end{array}\right), \hspace{0.5 cm} \tilde{A} = \left(\begin{array}{@{}c|c@{}}
  1 & \begin{matrix}
  0 & \ldots & 0 \\
  \end{matrix} \\
\hline
  \begin{matrix}
  0 \\
  \vdots \\
  0
  \end{matrix} & \tilde{A}'
\end{array}\right).$$

It is clear that $\rank A(p) = \rank A'(p) + 1$ and $\rank \tilde{A}(p) = \rank \tilde{A}'(p)+1$. Consequently, we have the relations $D_i = D'_{i-1} \times \C$ and $\tilde{D}_i = \tilde{D}'_{i-1} \times \C$ for $i \geq 1$. Observe that $D_i = \varnothing$ if and only if $D'_{i-1} = \varnothing$ and similarly for $\tilde{D}_i$ and $\tilde{D}'_{i-1}$. Note also that $D_0 = \tilde{D}_0 = \varnothing$. Thus, if they are not empty, $\dim D_i \leq k$ if and only if $\dim D'_{i-1} \leq k-1$ and $\dim \tilde{D}_i \leq k$ if and only if $\dim \tilde{D}'_{i-1} \leq k-1$. This proves the result for weakly Saito-holonomic and Saito-holonomic divisors if we take $k=i$ by virtue of Theorem \ref{ThCaractSH}. Taking $k=i-1$ and noticing that if $\dim \tilde{D}_1 = \dim(\tilde{D}_0' \times \C) \leq 0$, then $\tilde{D}_0' = \varnothing$, we prove the result for strong Saito-holonomicity.

\end{proof}

\begin{Rmk}
    An inductive argument shows that if $D = D' \times \C^k$, where $D'$ is a divisor in $\C^{n-k}$, then $D_i = D'_{i-k} \times \C^k$ and $\tilde{D}_i = \tilde{D}'_{i-k} \times \C^k$ for $i \geq k$, and $D_i = \tilde{D}_i = \varnothing$ for $i < k$.
\end{Rmk}

A direct algebraic translation of Theorem \ref{ThCaractSH} is the following:

\begin{Cor}
\label{CorAlgCaractSH}
    Let $D = V(f)$ be a germ of divisor in $(\C^n,0)$. Then:

    \begin{itemize}
        \item[(1)] $D$ is \emph{weakly Saito-holonomic} if and only if $\hspace{0,1 cm} \htt \Fitt_i(\mathcal{J}_f) \geq i$ for all $i=3, \ldots, n$.

        \item[(2)] $D$ is \emph{Saito-holonomic} if and only if $\hspace{0,1 cm} \htt \Fitt_i(\mathcal{J}_f/(f)) \geq i+1$ for all $i=2, \ldots, n-1$.

        \item[(3)] $D$ is \emph{strongly Saito-holonomic} if and only if $\hspace{0,1 cm} (\tilde{I}_1 =) \Fitt_n(\mathcal{J}_f) = \hol$ and $\htt \Fitt_i(\mathcal{J}_f) \geq i+1$ for all $i=2, \ldots, n-1$.
        
    \end{itemize}
    
\end{Cor}

\begin{proof}
    To prove $(1)$ and $(3)$ we use that $\tilde{D}_i = V(\tilde{I}_{i+1}) = V(\Fitt_{n-i}(\mathcal{J}_f))$, so $\dim \tilde{D}_i \leq k$ if and only if $\htt \Fitt_{n-i}(\mathcal{J}_f) \geq n-k$. We proceed similarly for $(2)$ with $D_i = V(I_{i+1}) = V(\Fitt_{n-i-1}(\mathcal{J}_f/(f)))$.

\end{proof}

In the free case, we know that, if a divisor is Koszul-free, then it is strongly Koszul-free if and only if it is strongly Euler-homogeneous \cite[Theorem 4.7]{NarvDual}. We are now able to generalize this result to the non-free case:

\begin{Th}
\label{ThSHandSEH}
    Let $D \subset X$ be a Saito-holonomic divisor. Then, it is strongly Saito-holonomic if and only if it is strongly Euler-homogeneous.
\end{Th}

\begin{proof}    
    These properties being local, we may assume $D$ is a germ of divisor in $(\C^n,0)$. If $D$ is strongly Euler-homogeneous, then $\tilde{D}_0 = \varnothing$ and $D_{i-1} = \tilde{D}_i$ for all $i=1, \ldots, n-1$. Since $D$ is also Saito-holonomic, $\dim D_i \leq i$ for all $i=0, \ldots, n-3$, so $\dim \tilde{D}_i = \dim D_{i-1} \leq i-1$ for all $i=1, \ldots, n-2$ and $D$ is strongly Saito-holonomic. 
    
    To prove the converse, we will proceed by induction on $n$. For $n=2$, by Theorem \ref{ThCaractSH}, a divisor is strongly Saito-holonomic if and only if $\tilde{D}_0 = \varnothing$. By Theorem \ref{ThGeomEH}, this happens if and only if it is strongly Euler-homogeneous. Suppose the result is true for $n-1$ and consider a strongly Saito-holonomic divisor $D$ in $(\C^n, 0)$. Since strong Saito-holonomicity is well-behaved with respect to smooth factors (Proposition \ref{PropProductSH}), by Remark \ref{RmkProductInduction}, we only have to show that $\tilde{D}_0 = \varnothing$. But this is guaranteed for a strongly Saito-holonomic divisor by Theorem \ref{ThCaractSH}.
 
\end{proof}

\begin{Rmk}
There are plenty of results in \cite{Uli17} in which the divisor is assumed to be Saito-holonomic and strongly Euler-homogeneous. By Theorem \ref{ThSHandSEH}, all of them can be reformulated in terms of strong Saito-holonomicity.
\end{Rmk}

In the free case, it is known that being of linear Jacobian type is equivalent to being strongly Koszul-free \cite[Proposition 1.11]{NarvDual}. In the non-free case, we only have the following:

\begin{Prop}
    Let $D \subset X$ be a divisor of linear Jacobian type. Then, it is strongly Saito-holonomic.
\end{Prop}

\begin{proof}
    We may suppose $D$ is a germ of divisor in $(\C^n,0)$. As $D$ is of linear Jacobian type, we have that $I:=\ker^{(1)} \varphi = \ker \varphi$. But $\ker \varphi$ is a prime ideal of height $n$ and $f \notin \ker \varphi$, so $\htt (I + (f)) = \htt (f, \sigma(\delta_1)-\alpha_1 s, \ldots, \sigma(\delta_m)-\alpha_m s) = n+1$. Thus, $D$ is strongly Saito-holonomic. 
    
\end{proof}

\begin{Ex}
    The converse is not true. Consider the hyperplane arrangement $D=V(f)$, where $f=xyz(x+t)(y+t)(z+t)(x+y+t)(x+z+t)(y+z+t)$ in $(\C^4,0)$. It is well-known that all hyperplane arrangements are Saito-holonomic and obviously strongly Euler-homogeneous. Thus, they are strongly Saito-holonomic by Theorem \ref{ThSHandSEH}. However, $D$ is an example of a hyperplane arrangement that is not of differential linear type (that is, the $\D[s]$-annihilator of $f^s$ is not generated by operators of order one) \cite[Example 5.7]{Uli17}, so it cannot be of linear Jacobian type \cite[Proposition 1.15]{CN09}.
\end{Ex}

Nevertheless, if we impose an extra condition on $\ker^{(1)} \varphi$, we have the converse:

\begin{Prop}
\label{PropUnmixedLJT}
    Let $D$ be a germ of strongly Saito-holonomic divisor in $(\C^n,0)$. If $\ker^{(1)} \varphi$ is unmixed (that is, all  its associated primes have the same height), then $D$ is of linear Jacobian type.
\end{Prop}

\begin{proof}
    As $D$ is strongly Saito-holonomic, in particular, it is weakly Saito-holonomic. Thus, $\htt (f, \sigma(\delta_1)-\alpha_1 s, \ldots, \sigma(\delta_m)-\alpha_m s) = n+1$ and $\htt (\sigma(\delta_1)-\alpha_1 s, \ldots, \sigma(\delta_m)-\alpha_m s) = n$. That is, $\htt ((f)+I) = n+1$ and $\htt I = n$, where $I=\ker^{(1)} \varphi$. A divisor is always of linear Jacobian type at smooth points, so the quotient $\ker \varphi/\ker^{(1)} \varphi$ is supported at the singular locus. In particular, for any $a \in \ker \varphi$, there exists $k \geq 1$ such that $f^k a \in I$. Thus, it remains to prove that $f$ is not a zero divisor modulo $I$. If that were the case, then $f$ would lie in some associated prime $\mathfrak{p}$ of $I$. In particular, $(f)+I \subset \mathfrak{p}$. But then $ \htt \mathfrak{p} \geq \htt ((f)+I) = n+1$, which is a contradiction with the fact that $I$ is unmixed.
    
\end{proof}

\begin{Rmk}
$ $
\begin{itemize}
    \item[(a)] The unmixedness theorem, which holds in any Cohen-Macaulay ring, states that any ideal of height $h$ that can be generated by $h$ elements is unmixed \cite[16.C]{MatsumuraRosa}. Thus, Proposition \ref{PropUnmixedLJT} applies when $\ker^{(1)} \varphi$ can be generated by $n$ elements. This is what happens in the free case.

    This result suggests the introduction of stronger versions of the Saito-holonomic type properties, in which we also require that the ideal in question is unmixed. Those notions would coincide with the previous ones for free divisors, for which the unmixedness hypothesis is automatic. However, we have not explored the relations between them.
    
    \item[(b)] Proposition \ref{PropUnmixedLJT} complements \cite[Corollary 3.23]{Uli17} where, instead of unmixedness, tameness is assumed.
    
\end{itemize}
    
\end{Rmk}

Let us revisit Examples \ref{ExSEH} and \ref{ExNotSEH} to study their Saito-holonomic type properties:

\begin{Ex}
Let $D = V(f)$ in $(\C^4,0)$ with $f=xy(x+z)(x^2+yz)(z+yt)$. We have that

$$\begin{array}{l}

    \tilde{D}_0 = \varnothing, \\[0,2 cm]

    \tilde{D}_1 = D_0 = V(x,y,z), \\[0,2 cm]

    \tilde{D}_2 = D_1 = V(x, z, yt). \\[0,2 cm]
\end{array}$$

As $\dim D_0 = 1$, $D$ is not Saito-holonomic. Consequently, it is not strongly Saito-holonomic. However, $\dim \tilde{D}_i \leq i$ for $i = 0,1$, so it is weakly Saito-holonomic.

Recall that $D$ is strongly Euler-homogeneous. This example tells us that we need to demand Saito-holonomicity in Theorem \ref{ThSHandSEH}. Weak Saito-holonomicity is not enough to guarantee the equivalence between strong Euler-homogeneity and strong Saito-holonomicity.

\end{Ex}

\begin{Ex}
Let $D = V(f)$ in $(\C^4,0)$ with $f=xy(x+z)(x^2+yz)(z^2+yt)$. We have that:

$$\begin{array}{l}

    \tilde{D}_0 = \varnothing, \\[0,2 cm]
    \tilde{D}_1 = V(x,y,z), \hspace{0,3 cm} D_0 = (x,y,z,t), \\[0,2 cm]
    \tilde{D}_2 = D_1 = V(t(y+z), t(x+z), z(z-t), z(y+t), z(x+t), xy - zt).
\end{array}$$

As $\dim D_i = i$ for $i=0,1$, $D$ is Saito-holonomic (and, therefore, also weakly Saito-holonomic). Since $D$ is not strongly Euler-homogeneous, Theorem \ref{ThSHandSEH} tells us that $D$ is not strongly Saito-holonomic either. And, indeed, we confirm this by noticing that $\dim \tilde{D}_1 = 1 > 0$. \\

\end{Ex}

\section{Extending the criteria to formal power series}
\label{SectionConvToFormal}

The objective of this section is to extend the characterizations developed in previous sections, in which the Saito matrices play an important role, to the case of formal power series. This will be useful to prove the main results of the next section.

Let $f \in \hol$. Take a generating set $\mathcal{T}=\{\hat{\delta}_1, \ldots, \hat{\delta}_m\}$ of $\what{\Der}_f$ and let $B = (\hat{\delta}_i(x_j))$ be the (formal) Saito matrix with respect to $\mathcal{T}$. Let $\hat{\alpha}_i \in \what{\hol}$ be such that $\hat{\delta}_i(f) = \hat{\alpha}_i f$ and define similarly $\tilde{B} = \left(B \hspace{0,1 cm}\Big| -\overline{\hat{\alpha}}\right)$, the (formal) extended Saito matrix with respect to $\mathcal{T}$ and $f$. 

Consider the formal Jacobian ideal $\what{\mathcal{J}}_f = \what{\hol}(\partial_1(f), \ldots, \partial_n(f), f)$. Let $J_i = \Fitt_{n-i}(\what{\mathcal{J}}_f/(f))$. As in the convergent case, a free presentation of $\what{\mathcal{J}}_f/(f)$ as an $\what{\hol}$-module is given by the matrix $B$, so $J_i$ is the ideal of $\what{\hol}$ generated by the minors of order $i$ of $B$ (for $1 \leq i \leq n$). Similarly, if we set $\tilde{J}_i = \Fitt_{n+1-i}(\what{\mathcal{J}}_f)$, then $\tilde{J}_i$ is the ideal of $\what{\hol}$ generated by the minors of order $i$ of $\tilde{B}$.

Once more, these ideals do not depend on the choice of basis, coordinates or equation. Note that any convergent generating set $\mathcal{S}$ of $\Der_f$ is still an $\what{\hol}$-generating set of $\what{\Der}_f$, so we deduce that $J_i$ and $\tilde{J}_i$ are also the ideals of $\what{\hol}$ generated by the minors of order $i$ of the Saito matrices with respect to $\mathcal{S}$ and $f$. That is, $J_i = I_i^e$ and $\tilde{J}_i = \tilde{I}_i^e$. 

In particular, by $(3)$ in Proposition \ref{PropDIrel}, we have that $\sqrt{J_{n-1}} = \sqrt{\tilde{J}_n} = \sqrt{\what{\mathcal{J}}_f}$.

Now, we can state the algebraic criterion for strong Euler-homogeneity in terms of $J_i$ and $\tilde{J}_i$:

\begin{Prop}[Formal criterion for strong Euler-homogeneity]
\label{PropFormalAlgEH}
    Let $D = V(f)$ be a germ of divisor in $(\C^n,0)$ with $n \geq 2$. Then:

    \begin{itemize}
        \item[(1)] $D$ is strongly Euler-homogeneous on $D_0$ if and only if $\tilde{J}_1 = \what{\hol}$.
        \item[(2)] $D$ is strongly Euler-homogeneous on $D \setminus D_0$ if and only if $\sqrt{\tilde{J}_{i+1}} = \sqrt{J_i}$ \text{ for } $i=1, \ldots, n-2$.
        \item[(3)] $D$ is strongly Euler-homogeneous if and only if $\tilde{J}_1 = \what{\hol}$ and $\sqrt{\tilde{J}_{i+1}} = \sqrt{J_i}$ \text{ for } $i=1, \ldots, n-2$.
    \end{itemize}
\end{Prop}

\begin{proof}
    Due to faithful flatness of $\what{\hol}$ over $\hol$, given two ideals $I,J \subset \hol$, we have that $\sqrt{I} = \sqrt{J}$ if and only if $\sqrt{I^e} = \sqrt{J^e}$ in $\what{\hol}$. Thus, $\sqrt{\tilde{I}_{i+1}} = \sqrt{I_i}$ if and only if $\sqrt{\tilde{J}_{i+1}} = \sqrt{J_i}$. The result follows from Corollary \ref{CorAlgEH}.
    
\end{proof}

\begin{Rmk}
    As before, this criterion can be read in terms of Fitting ideals. We just need to replace $\mathcal{J}_f$ by $\what{\mathcal{J}}_f$ and $\mathcal{J}_f/(f)$ by $\what{\mathcal{J}}_f/(f)$ in Remark \ref{RmkSEHFitting}.
\end{Rmk}

The relevance of Proposition \ref{PropFormalAlgEH} is that, unexpectedly, formal power series, which only make sense at the origin, can be used to study strong Euler-homogeneity away from this point. Furthermore, the necessary condition for having strong Euler-homogeneity outside $D_0$ (Corollary \ref{CorAlphaDtilde0}) also has a formal analogue, although the proof is not so straightforward: 

\begin{Cor}
\label{CorSEHoutsideD0Formal}
    Let $D=V(f)$ be a germ of divisor in $(\C^n,0)$ that is strongly Euler-homogeneous outside $D_0$. Let $g$ be a ``formal equation'' of $D$ (i.e. a formal unit multiplied by $f$). Then $\sqrt{\tilde{J}_1} = \sqrt{\partial g : g}$.
\end{Cor}

\begin{proof}

Let $\{\hat{\eta}_1, \ldots, \hat{\eta}_m\}$ be a generating set of $\what{\Der}_f$ such that $\hat{\eta}_i(g) = \hat{\beta}_i g$ for all $i=1, \ldots, m$, so that $\partial g : g = (\hat{\beta}_1, \ldots, \hat{\beta}_m)$. Recall that the $J_i$ (resp. $\tilde{J}_i$) do not depend on the chosen presentation, so these ideals are the ones generated by the minors of order $i$ of the Saito matrix with respect to $\{\hat{\eta}_1, \ldots, \hat{\eta}_m\}$ (resp. the extended Saito matrix with respect to $\{\hat{\eta}_1, \ldots, \hat{\eta}_m\}$ and $g$). Thus, $(\hat{\beta}_1, \ldots, \hat{\beta}_m) \subset \tilde{J}_1$ and it is clear that $\sqrt{(\hat{\beta}_1, \ldots, \hat{\beta}_m)} \subset \sqrt{\tilde{J}_1}$. 

Let us prove the other inclusion. The proof starts as in Lemma \ref{LemAlphaSubsetD}: $g$ is integral over the ideal of its partial derivatives (convergent series with the trivial metric are exactly formal power series, so the proof of \cite[Satz 5.2]{Scheja70} is still valid). Thus, either some $\hat{\beta}_i$ is a unit or a power of $g$ belongs to $(\hat{\beta}_1, \ldots, \hat{\beta}_m)$. In the first case, it is clear that $(\hat{\beta}_1, \ldots, \hat{\beta}_m) = \tilde{J}_1 = \what{\hol}$ and we are done. In the second one, as $D_{n-1} = D$, $\sqrt{I_n} = \sqrt{(f)}$ and $\sqrt{J_n} = \sqrt{I_n^e} = \sqrt{(f)^e} = \sqrt{(g)} \subset \sqrt{(\hat{\beta}_1, \ldots, \hat{\beta}_m)}$.

Note that, reasoning as in $(1)$ in Proposition \ref{PropDIrel}, we have $\tilde{J}_i \subset J_i + (\hat{\beta}_1, \ldots, \hat{\beta}_m)$, so $\sqrt{\tilde{J}_i} \subset \sqrt{J_i + (\hat{\beta}_1, \ldots, \hat{\beta}_m)} = \sqrt{\sqrt{J_i} + \sqrt{(\hat{\beta}_1, \ldots, \hat{\beta}_m)}}$ for all $i$. In particular, as $\sqrt{J_n} \subset \sqrt{(\hat{\beta}_1, \ldots, \hat{\beta}_m)}$, we also have $\sqrt{\tilde{J}_n} \subset \sqrt{(\hat{\beta}_1, \ldots, \hat{\beta}_m)}$. But, as $D$ is strongly Euler-homogeneous outside $D_0$, we have that $\sqrt{J_i} = \sqrt{\tilde{J}_{i+1}}$ for $i \leq n-1$ (Proposition \ref{PropFormalAlgEH}). This implies $\sqrt{J_{n-1}} = \sqrt{\tilde{J}_n} \subset \sqrt{(\hat{\beta}_1, \ldots, \hat{\beta}_m)}$ and so

$$\sqrt{\tilde{J}_{n-1}} \subset \sqrt{\sqrt{J_{n-1}} + \sqrt{(\hat{\beta}_1, \ldots, \hat{\beta}_m)}} \subset \sqrt{(\hat{\beta}_1, \ldots, \hat{\beta}_m)}.$$

By reverse induction, we finally get $\sqrt{\tilde{J}_1} \subset \sqrt{(\hat{\beta}_1, \ldots, \hat{\beta}_m)}$. 
    
\end{proof}

\begin{Rmk}
    As in the convergent case, we deduce that, when $D$ is strongly Euler-homogeneous outside $D_0$, the radical of the ideal $\partial g : g$ does not depend on the choice of formal equation $g$.
\end{Rmk}

We have argued that the Fitting ideals of the $\what{\hol}$-modules $\what{\mathcal{J}}_f$ and $\what{\mathcal{J}}_f/(f)$ are the extended ones of those of $\mathcal{J}_f$ and $\mathcal{J}_f/(f)$. Since the height of the extension is the same, we deduce from Corollary \ref{CorAlgCaractSH} that we can also check (weak or strong) Saito-holonomicity using formal data:

\begin{Cor}
\label{CorFormalAlgCaractSH}
    Let $D = V(f)$ be a germ of divisor in $(\C^n,0)$. Then:

    \begin{itemize}
        \item[(1)] $D$ is \emph{weakly Saito-holonomic} if and only if $\hspace{0,1 cm} \htt \Fitt_i(\what{\mathcal{J}}_f) \geq i$ for all $i=3, \ldots, n$.

        \item[(2)] $D$ is \emph{Saito-holonomic} if and only if $\hspace{0,1 cm} \htt \Fitt_i(\what{\mathcal{J}}_f/(f)) \geq i+1$ for all $i=2, \ldots, n-1$.

        \item[(3)] $D$ is \emph{strongly Saito-holonomic} if and only if $\hspace{0,1 cm} (\tilde{J}_1 = ) \Fitt_n(\what{\mathcal{J}}_f) = \what{\hol}$ and $\htt \Fitt_i(\what{\mathcal{J}}_f) \geq i+1$ for all $i=2, \ldots, n-1$.
        
    \end{itemize}
    
\end{Cor}

\vspace{0,2 cm}

\section{Proving new cases of Conjecture \ref{ConjLCT-SEH}}
\label{SectionConj}

In this section we prove Conjecture \ref{ConjLCT-SEH} in three different cases: when the divisor is strongly Euler-homogeneous everywhere except on a discrete set of points, when the divisor is weakly Koszul-free and when the ambient dimension is four.

\subsection{Assuming strong Euler-homogeneity outside a discrete set of points}
\label{SubsectionPuncturedBall}

In general, there is no hope that strong Euler-homogeneity outside a point implies that the point also has this property. In order to give a counterexample, note that a divisor is always strongly Euler-homogeneous at smooth points (because a local equation is $x_1 = 0$, which verifies $x_1 \partial_1 (x_1) = x_1$). So we just need to consider an isolated singularity that is not quasihomogeneous (and so not strongly Euler-homogeneous by \cite[Satz 4.1]{Saito71}) at the origin. Take, for example, $V(x^4+y^5+y^4x) \subset \C^2$. However, we will see that, for free divisors and under mild assumptions, this becomes true.

\begin{Th}
\label{ThSEHoutside0}
    Let $D=V(f)$ be a germ of free divisor in $(\C^n,0)$ that is strongly Euler-homogeneous outside the origin. Then, the following are equivalent:

    \begin{itemize}
        \item[(1)] $D$ is strongly Euler-homogeneous at the origin.
        \item[(2)] There exists a non-topologically nilpotent singular derivation in $\what{\Der}_f$.
        \item[(3)] There exists a non-zero semisimple derivation in $\what{\Der}_f$.
    \end{itemize}
\end{Th}

\begin{proof}
$ $ \\[0,2 cm]
    $(1) \Rightarrow (2)$: If $D$ is strongly Euler-homogeneous at $0$, then there exists a singular derivation $\delta \in \what{\Der}_f$ such that $\delta(f) = f$. By statement $(2)$ in Proposition \ref{PropSNdVR24}, $\delta$ cannot be topologically nilpotent and we are done.

    \vspace{0,2 cm}

    \noindent $(2) \Rightarrow (3)$: If there exists a non-topologically nilpotent singular derivation $\delta \in \what{\Der}_f$, in particular, its semisimple part $\delta_S$ is not zero and it is still logarithmic (Proposition \ref{PropSNdVR24}, $(3)$).

    \vspace{0,2 cm}
    
    \noindent $(3) \Rightarrow (1)$: Suppose $D$ is not strongly Euler-homogeneous at $0$. In particular, it cannot be a product at $0$ (otherwise, it would be non-strongly Euler-homogeneous on a set of positive dimension), so $\what{\Der}_f \subset \mm \what{\Der}$ by $(3)$ in Proposition \ref{PropForm}. On the other hand, by Proposition  \ref{PropRank}, the origin must be the only point $p$ at which $\rank A(p) = \rank \tilde{A}(p) = 0$. That is, $\tilde{D}_0 = \{0\}$ and $\mm = \sqrt{\tilde{I}_1}$. Taking extensions, we have $\what{\mm} = \sqrt{\tilde{J}_1}$. 

      If there exists a non-zero semisimple derivation $\delta \in \what{\Der}_f$, then its class in $\what{\Der}_f/\mm \what{\Der}_f$ is non-zero. Indeed, if $\delta \in \mm \what{\Der}_f \subset \mm^2 \what{\Der}$, then $\delta$, not having linear part, would also be topologically nilpotent by $(1)$ in Proposition \ref{PropSNdVR24}, something that can only happen if $\delta = 0$ by uniqueness. 
      
      Thus, by Nakayama's Lemma, we can construct a formal basis $\{\delta_1, \ldots, \delta_n\}$ of $\what{\Der}_f$ with $\delta_1 = \delta$. Now, we can find a formal unit $u$ such that, for $g=uf$, we have $\delta_1(g) = \beta g$ with $\beta \in \C$ (Proposition \ref{PropSNdVR24}, $(4)$). Since we are assuming $D$ is not strongly Euler-homogeneous at the origin, $\beta = 0$ (otherwise, $\tilde{J}_1 = \what{\hol}$ and $D$ would be strongly Euler-homogeneous at $0$ by Proposition \ref{PropFormalAlgEH}). Let $\beta_i \in \what{\hol}$ be such that $\delta_i(g) = \beta_i g$ (note that $\beta_1 = \beta = 0$). By Corollary \ref{CorSEHoutsideD0Formal}, we have $\what{\mm} = \sqrt{\tilde{J}_1} =$ $\sqrt{(\beta_1, \ldots, \beta_n)} = \sqrt{(\beta_2, \ldots, \beta_n)}$, but this is impossible since $\htt \what{\mm} = n$ and $ \htt \sqrt{(\beta_2, \ldots, \beta_n)} = \htt (\beta_2, \ldots, \beta_n) \leq n-1$.
    
\end{proof}

\begin{Rmk}
$ $
\begin{itemize}
    \item[(a)] In particular, if $D$ is a non-quasihomogeneous germ of plane curve (which is an isolated singularity), then all its logarithmic derivations are topologically nilpotent. This generalizes \cite[Proposition 5.2]{dVR24}, where we proved that, for plane curves $D=V(f)$ with $f \in \mm^3$, non-Euler logarithmic derivations (those $\delta$ for which $\delta(f) \in \what{\mm} f$) are topologically nilpotent.
    
    \item[(b)] As we said before, $D=V(f)$ with $f=x^4+y^5+y^4x$ is a non-quasihomogeneous plane curve. A basis of $\Der_f$ is:

    $$\begin{array}{l}

    \delta_1 = (4x^2+5xy) \partial_x + (3xy+4y^2) \partial_y, \\

    \delta_2 = (16xy^2+4y^3-125xy) \partial_x + (12y^3-4x^2+5xy-100y^2) \partial_y. \\

    \end{array}$$
    
    And, indeed, as they have no linear part, every logarithmic derivation is topologically nilpotent.
\end{itemize}
\end{Rmk}

A direct consequence of Theorem \ref{ThSEHoutside0} is that a weak version of Conjecture \ref{ConjLCT-SEH} holds in arbitrary dimension:

\begin{Th}
\label{ThConjOutsideDiscreteSet}
    Let $D \subset X$ be a free divisor in a complex analytic manifold of dimension $n$ satisfying LCT and being strongly Euler-homogeneous outside a discrete set of points. Then $D$ is strongly Euler-homogeneous.
\end{Th}

\begin{proof}
    As LCT and strong Euler-homogeneity are local properties, we may assume that $D$ is a germ of divisor in $(\C^n,0)$ satisfying LCT and being strongly Euler-homogeneous outside $0$. Suppose $D$ is not strongly Euler-homogeneous at the origin. By Theorem \ref{ThSEHoutside0}, all singular logarithmic derivations are topologically nilpotent. In particular, all of them have zero trace so, by Theorem \ref{ThZeroTrace}, LCT cannot hold.
    
\end{proof}

\begin{Rmk}
    This gives an alternative proof of the conjecture for $n=2$, since singularities are isolated and a divisor is always strongly Euler-homogeneous at smooth points. \\
\end{Rmk}

\subsection{Assuming weak Koszul-freeness}
\label{SubsectionWK}

The characterization given in Theorem \ref{ThCaractSH} allows us to prove that the conjecture holds for weakly Koszul-free divisors:

\begin{Th}
\label{ThConjWKoszul}
    Let $D \subset X$ be a weakly Koszul-free divisor. If $D$ satisfies LCT, then it is strongly Euler-homogeneous.
\end{Th}

\begin{proof}
    Since these are local properties, we may assume $D$ is a germ of free divisor in $(\C^n, 0)$. The statement is trivially true for $n = 1$ (we also know that it is true for $n=2$, but the following argument gives a new proof of this fact). Suppose it is true for dimension $n$ and let us prove it for dimension $n+1$. 
    
    If $D$ is a product at a point, then $D$ is locally isomorphic to $D' \times \C$, where $D'$ is a germ of weakly Koszul-free divisor in $\C^n$ that satisfies LCT by Proposition \ref{PropProductSH} and Lemma \ref{LemProd}. By induction hypothesis, $D'$ is strongly Euler-homogeneous and so is $D$. This proves that $D$ is strongly Euler-homogeneous outside $D_0$ (where $D$ is a product). As $D$ is weakly Koszul-free, by Theorem \ref{ThCaractSH}, $\dim \tilde{D}_0 \leq 0$, so $\tilde{D}_0 \subset \{0\}$. By Proposition \ref{PropRank}, $D$ is always strongly Euler-homogeneous on $D_0 \setminus \tilde{D}_0$, so we have strong Euler-homogeneity on a punctured neighbourhood of $0$. Since $D$ satisfies LCT, by Theorem \ref{ThConjOutsideDiscreteSet}, $D$ must also be strongly Euler-homogeneous at $0$.
    
\end{proof}

It is well-known that, even for free divisors, LCT does not imply Koszul-freeness. The four lines is a classical counterexample. However, this divisor is still weakly Koszul-free. From Theorem \ref{ThCaractSH} we can easily deduce an interesting consequence: this is a general fact for free divisors in ambient dimension $n=3$.

\begin{Cor}
\label{CorWKoszulLCT}
Let $D \subset X$ be a free divisor in a complex analytic manifold of dimension $3$. If $D$ satisfies LCT, then it is weakly Koszul-free.
\end{Cor}

\begin{proof}
    These properties being local, we may assume $D$ is a germ of free divisor in $(\C^3,0)$. By Theorem \ref{ThCaractSH}, $D$ is weakly Koszul-free if and only if $\dim \tilde{D}_0 \leq 0$. As $D$ satisfies LCT, it is strongly Euler-homogeneous (because we know Conjecture \ref{ConjLCT-SEH} is true for $n=3$). In particular, $\tilde{D}_0 = \varnothing$ and we are done.
    
\end{proof}

\subsection{Assuming ambient dimension $n=4$}
\label{SubsectionDim4}

M. Granger and M. Schulze proved in \cite[Theorem 5.4]{GS} the so-called \emph{formal structure theorem for logarithmic vector fields}. Given a germ $f \in \what{\hol}$, they call $s$ the maximal dimension of the vector space of diagonal derivations $\sigma \in \Der_f$ such that $\sigma(f) \in \C \cdot f$ for $f$ varying in a formal contact equivalence class (that is, $s$ is maximal for all formal coordinate systems and changes of $f$ by unit factors). Then, they prove that there exists a minimal generating set of $\Der_f$ in which $s$ derivations are diagonal in some formal coordinate system and the rest of them are (topologically) nilpotent satisfying some additional properties. 

They use the notion of semisimplicity given by K. Saito (for them, a derivation is semisimple if it only has linear part and the associated matrix is semisimple), which depends on the choice of coordinates. In this subsection, we first develop a version of this theorem in which we use the concepts of semisimplicity and topological nilpotency from Definition \ref{DefSN}, following the Gérard-Levelt approach. This allows us to define $s$ intrinsically as the maximal number of commuting semisimple derivations in a minimal generating set of $\Der_f$, in such a way that it is completely independent of the choice of equation or coordinate system. Then, we use this version of the formal structure theorem to show that Conjecture \ref{ConjLCT-SEH} holds unconditionally in the $4$-dimensional case. 

First, we need to generalize the statement $(4)$ in Proposition \ref{PropSNdVR24} to the case of multiple diagonal derivations:

\begin{Prop}
\label{PropDiagDer}
Let $0 \neq f \in \hol$, let $\delta_1, \dots, \delta_m \in \mm \what{\Der}$ be logarithmic derivations for $f$ that are diagonal in some fixed formal coordinate system and let $\alpha_1, \ldots, \alpha_m \in \what{\hol}$ be such that $\delta_i(f) = \alpha_i f$ for all $i=1, \ldots, m$. Then, there exists a formal unit $u$ such that $g=uf$ verifies $\delta_i(g) = \alpha_{i0} g$ for all $i=1, \ldots, m$, where $\alpha_{i0}$ is the constant term of $\alpha_i$.
\end{Prop}

\begin{proof}

Let us write $\delta_i = \underline{x} D_i \overline{\partial}$, where $D_i = \diag(\lambda_i)$ and $\lambda_i$ is a vector of (possibly complex) weights. Also write $\alpha_i = \sum_{\beta \in \N^n} \alpha_{i\beta}{x}^{\beta}$.

Now, set

$$ u = \exp \left( - \sum_{{\beta \cdot \lambda_1} \neq 0} \dfrac{\alpha_{1\beta} x^\beta}{\beta \cdot \lambda_1} - \ldots -  \sum_{\substack{\beta \cdot \lambda_1 = 0 \\ \cdots \\ \beta \cdot \lambda_m \neq 0}} \dfrac{\alpha_{m\beta} x^\beta}{\beta \cdot \lambda_m}\right).  $$

\vspace{0,3 cm}

Note that $u$ is a well-defined formal power series because the argument of the exponential, call it $b$, vanishes at the origin ($\beta = 0$ is not in the sum), so we can compose the two series and the result is another formal power series that does not vanish at the origin. Thus, $u$ is a formal unit and

$$\delta_i(u) = \delta_i(e^b) = e^b \delta_i(b) = u \delta_i(b).$$

\vspace{0,2 cm}

In order to compute $\delta_i(b)$, take into account that $\delta_i$ applied to a monomial $x^\beta$ returns $(\beta \cdot \lambda_i) x^\beta$. Moreover, diagonal derivations always commute, so $0=[\delta_i, \delta_j](f) = (\delta_i(\alpha_j) - \delta_j(\alpha_i)) f$ and we deduce that $\delta_i(\alpha_j) = \delta_j(\alpha_i)$. Therefore:

$$ \delta_i(\alpha_j) = \sum_{\beta} (\beta \cdot \lambda_i) \alpha_{j\beta} x^\beta = \sum_{\beta} (\beta \cdot \lambda_j) \alpha_{i\beta} x^\beta = \delta_j(\alpha_i), $$

\noindent that is, $(\beta \cdot \lambda_i) \alpha_{j\beta} = (\beta \cdot \lambda_j) \alpha_{i\beta}$ for all $i,j =1, \ldots, m$ and $\beta \in \N^n$. Thus,

\begin{align*}
    \delta_i(b) &= -\sum_{{\beta \cdot \lambda_1} \neq 0} \dfrac{\alpha_{1\beta} (\beta \cdot \lambda_i) }{\beta \cdot \lambda_1} x^\beta - \ldots -  \sum_{\substack{\beta \cdot \lambda_1 = 0 \\ \cdots \\ \beta \cdot \lambda_i \neq 0}} \dfrac{\alpha_{i\beta} (\beta \cdot \lambda_i)}{\beta \cdot \lambda_i} x^\beta \\ &= -\sum_{{\beta \cdot \lambda_1} \neq 0} \alpha_{i\beta} x^\beta - \ldots -  \sum_{\substack{\beta \cdot \lambda_1 = 0 \\ \cdots \\ \beta \cdot \lambda_i \neq 0}} \alpha_{i\beta} x^\beta.
\end{align*}

Let $g=uf$, so that 

\begin{align*}
    \delta_i(g) &= \delta_i(u) f + u \delta_i(f) = u \delta_i(b) f + u \alpha_i f = (\alpha_i+\delta_i(b)) uf = (\alpha_i+\delta_i(b)) g \\ &= \left( \sum_{\beta} \alpha_{i\beta} x^\beta - \sum_{{\beta \cdot \lambda_1} \neq 0} \alpha_{i\beta} x^\beta - \ldots -  \sum_{\substack{\beta \cdot \lambda_1 = 0 \\ \cdots \\ \beta \cdot \lambda_i \neq 0}} \alpha_{i\beta} x^\beta \right) g = \left( \sum_{\substack{\beta \cdot \lambda_1 = 0 \\ \cdots \\ \beta \cdot \lambda_i = 0}} \alpha_{i\beta} x^\beta \right) g.
\end{align*}

\vspace{0,3 cm}

Let us see that $\alpha'_i := \displaystyle\sum_{\substack{\beta \cdot \lambda_1 = 0 \\ \cdots \\ \beta \cdot \lambda_i = 0}} \alpha_{i\beta} x^\beta$ must be $\alpha_{i0}$ and we will have the result:

\vspace{0,3 cm}

If we decompose $g$ as a sum of eigenvectors of $\delta_i$ (which can be done in a unique way by \cite[Lemma 2.3]{Saito71}), $g = \sum_{\mu \in \C} g_{\mu}$ with $\delta_i(g_\mu) = \mu g_\mu$, we have

\vspace{0,2 cm}

$$ \delta_i(g) = \sum_{\mu \in \C} \mu g_\mu = \alpha'_i g = \sum_{\mu \in \C} \alpha'_i g_\mu. $$

\vspace{0,3 cm}

But note that $\delta_i(\alpha'_i) = 0$ because each summand $x^\beta$ of $\alpha'_i$ satisfies $\beta \cdot \lambda_i = 0$. Then, $\delta_i(\alpha'_i g_\mu) = \alpha'_i \delta_i(g_\mu) = \mu \alpha'_i g_\mu$ and so $\mu g_\mu$ and $\alpha'_i g_\mu$ are both eigenvectors of $\delta_i$ for $\mu$. By the uniqueness of the decomposition, we must have $\mu g_\mu = \alpha'_i g_\mu$ for every $\mu \in \C$. As $g \neq 0$, some $g_\mu \neq 0$. But then $\alpha'_i = \mu$ is constant and the only possibility for this is $\alpha'_i = \alpha_{i0}$, as that is the value of $\alpha'_i$ at the origin.

\end{proof}

\begin{Th}
\label{ThFST}
    Let $f \in \what{\hol}$ be such that $\Der_f \subset \mm \what{\Der}$ and let $s \geq 0$ be the maximal number of commuting semisimple derivations in a minimal generating set of $\Der_f$. If $s=0$, then each minimal generating set of $\Der_f$ is formed by topologically nilpotent derivations. Otherwise, there exist a minimal generating set $\{\sigma_1, \ldots, \sigma_s, \nu_1, \ldots, \nu_r\}$ of $\Der_f$, a formal coordinate system $(y_1, \ldots, y_n)$ and a formal unit $u$ such that

    \begin{itemize}
        \item[(1)] $\sigma_1, \ldots, \sigma_s$ are diagonal with rational weights: $\sigma_i = \sum_{j=1}^n w_{ij} y_j \dfrac{\partial}{\partial y_j}$ with $w_{ij} \in \Q$ for all $i=1, \ldots, s$ and $j=1, \ldots, n$.
        \item[(2)] For $g=uf$ we have $\sigma_i(g) = \alpha_i g$ with $\alpha_i \in \Q$ for all $i=1, \ldots, s$.
        \item[(3)] $\nu_1, \ldots, \nu_r$ are topologically nilpotent.
        \item[(4)] $[\sigma_i, \nu_j] = c_{ij} \nu_j$ with $c_{ij} \in \Q$ (that is, $\nu_j$ is $\sigma_i$-homogeneous of weight $c_{ij}$) for all $i=1, \ldots, s$, $j=1, \ldots, r$.
    \end{itemize}
\end{Th}

\begin{proof}
     Suppose $s=0$ and let $\{\delta_1, \ldots, \delta_k\}$ be a minimal generating set of $\Der_f$. Their semisimple and topologically nilpotent parts are still logarithmic (Proposition \ref{PropSNdVR24}), so $\{\delta_{1S}, \delta_{1N}, \ldots, \delta_{kS},\delta_{kN}\}$ is a generating set. We claim that $\delta_{iS}=0$ for all $i$: the class of each $\delta_{iS}$ modulo $\mm \Der_f$ must be zero (otherwise, it could be extended to a basis of $\Der_f/\mm\Der_f$, giving, by Nakayama's Lemma, a minimal generating set of $\Der_f$ with a semisimple derivation and $s \geq 1$). Thus, $\delta_{iS} \subset \mm\Der_f \subset \mm^2 \what{\Der}$. But any derivation in $\mm^2 \what{\Der}$ is topologically nilpotent (as it has no linear part) so $\delta_{iS}$ is both semisimple and topologically nilpotent. By the uniqueness of the decomposition, we must have $\delta_{iS} = 0$. This means $\delta_i = \delta_{iN}$ is topologically nilpotent for all $i$ and we are done.

    Now suppose $s \geq 1$ and consider a minimal generating set $\{\sigma_1, \ldots, \sigma_s, \eta_1, \ldots, \eta_r\}$ of $\Der_f$, where $\sigma_1, \ldots, \sigma_s$ are commuting semisimple derivations. We claim that there exists a regular system of parameters $(y_1, \ldots, y_n)$ such that $\sigma_i = \sum_{j=1}^n w_{ij} y_j \partial_j$ with $w_{ij} \in \C$ for all $i=1, \ldots, s$, $j=1, \ldots, n$. We prove it by mimicking the proof of \cite[Théorème 2.3]{GL}:

    First, we need to show that each closed (for the $\what{\mm}$-adic topology) subspace $H \subset k[[x_1, \ldots, x_n]]$ that is stable by $\sigma_1, \ldots, \sigma_s$ admits a closed supplementary subspace that is also stable by $\sigma_1, \ldots, \sigma_s$. This is easily done by adapting the proof of \cite[Proposition 1.3]{GL}, using the well-known fact that commuting semisimple endomorphisms of a finite dimensional $\C$-vector space are simultaneously diagonalizable. The previous assertion shows the existence of a $\C$-vector space $W$ stable by $\sigma_1, \ldots, \sigma_s$ such that $\what{\mm} = \what{\mm}^2 \oplus W$. As $\dim W = n$, we can take a basis $\{y_1, \ldots, y_n\}$ of $W$ formed by common eigenvectors of $\sigma_1, \ldots, \sigma_s$. If we denote by $w_{ij}$ the eigenvalue of $y_j$ for $\sigma_i$, then $(y_1, \ldots, y_n)$ is our desired system of parameters and the proof of the claim is finished.

    In this formal coordinate system the $\sigma_i$ are diagonal, so we can apply Proposition \ref{PropDiagDer} to deduce that there exists a formal unit $u$ such that $g=uf$ verifies $\sigma_i(g) = \alpha_i g$ with $\alpha_i \in \C$ for all $i=1, \ldots, s$. Also, we may suppose that $w_{ij}$ and $\alpha_i$ are rational for all $i,j$ by \cite[Lemma 1.4]{Saito71}.
    
    Write $\underline{\sigma} = (\sigma_1, \ldots, \sigma_s)$. Now, we decompose $\eta_i = \sum_{\underline{\lambda} \in \Q^s} \eta_{i \underline{\lambda}}$, where $\eta_{i \underline{\lambda}}$ is $\underline{\sigma}$-homogeneous of degree $\underline{\lambda}$ for all $i=1, \ldots, r$. Let $\beta \in \what{\hol}$ be such that $\eta_i(g) = \beta g$ (recall that $\eta_i$ is also logarithmic for $g$) and write in a similar way $\beta = \sum_{\underline{\lambda} \in \Q^s} \beta_{\underline{\lambda}}$. We get 

     $$ \sum_{\underline{\lambda} \in \Q^s} \eta_{i \underline{\lambda}} (g) = \sum_{\underline{\lambda} \in \Q^s} \beta_{\underline{\lambda}} g, $$

     \vspace{0,2 cm}

    \noindent where $\eta_{i \underline{\lambda}} (g)$ and $\beta_{\underline{\lambda}} g$ are both $\underline{\sigma}$-homogeneous of the same degree ($\underline{\lambda} + \underline{\alpha}$, where we are writing $\underline{\alpha} = (\alpha_1, \ldots, \alpha_s)$). By the uniqueness of this decomposition, we get $\eta_{i \underline{\lambda}} (g) = \beta_{\underline{\lambda}} g$, so $\eta_{i \underline{\lambda}}$ is logarithmic for all $i=1, \ldots, r$ and $\underline{\lambda} \in \Q^s$. 

    Note that each derivation $\sigma_i \in \mm\what{\Der}$ induces a linear map $\overline{\sigma_i}$ in $\Der_f/\mm\Der_f$ by $\overline{\sigma_i}(\bar{\eta}) = \overline{[\sigma_i, \eta]}$ (it is well-defined since $\eta \in \mm\Der_f$ implies $[\sigma_i, \eta] \in \mm\Der_f$). We claim that, for each $\delta \in \Der_f$, the number of $\delta_{\underline{\lambda}}$ whose class in $\Der_f/\mm\Der_f$ is not zero is finite: suppose the class of $\delta_{\underline{\lambda}}$ in $\Der_f/\mm\Der_f$ is not zero. Then, it is a common eigenvector of the induced maps $\overline{\sigma_i}$ with $\underline{\lambda}$ as vector of eigenvalues. But since $\Der_f/\mm\Der_f$ is of finite dimension $s+r$, the number of different $s$-tuples of eigenvalues is also finite. 
    
    In particular, the class of each $\eta_i$ is a (finite) linear combination of the classes of the $\eta_{i\underline{\lambda}}$, so these classes together with those of the $\sigma_i$ (a finite number) generate $\Der_f/\mm\Der_f$. By Nakayama's Lemma, the corresponding derivations generate $\Der_f$. Thus, we get that a finite subset $S$ of $\{\sigma_1, \ldots, \sigma_s, \eta_{1 \underline{\lambda}}, \ldots, \eta_{r \underline{\lambda}}, \mid \underline{\lambda} \in \Q^s\}$ containing $\sigma_1, \ldots, \sigma_s$ is a generating set of $\Der_f$.

    Now, for each $\eta\in S\setminus\{\sigma_1, \ldots, \sigma_s\}$ we do the following: if its $\underline{\sigma}$-degree is $\underline{\lambda} \neq \underline{0}$, then it is topologically nilpotent \cite[Lemma 2.6]{GS} and we do nothing. If $\underline{\lambda} = \underline{0}$, then we replace $\eta$ by $\eta_S, \eta_N$ (which still commute with the $\sigma_i$ by \cite[Théorème 1.6]{GL}) and we still have a generating set of $\Der_f$. Denote again by $S$ the generating set obtained from this process (which is formed exclusively by $\underline{\sigma}$-homogeneous topologically nilpotent and commuting semisimple derivations) and extract a minimal generating set $S'$ from $S$ containing $\sigma_1, \ldots, \sigma_s$. Since $s$ is maximal, the only semisimple derivations in $S'$ must be $\sigma_1, \ldots, \sigma_s$. Thus, the rest of them, $\nu_1, \ldots, \nu_r$, must be topologically nilpotent and we get the desired result.
    
\end{proof}

\newpage

\begin{Rmk}
$ $
\begin{itemize}
    \item[(a)] If $f \in \hol$, Theorem \ref{ThFST} is valid for $\what{\Der}_f$.
    \item[(b)] As previously mentioned, this version of the formal structure theorem gives an intrinsic definition of the number $s$. Let us notice that M. Schulze gave in \cite[Theorem 2]{SchulzeFSTMaximalTorus} another characterization of this number as the dimension of a maximal torus in the Lie group of automorphisms $\varphi: \what{\hol} \to \what{\hol}$ such that $\varphi(f) \in (f)$.
\end{itemize}
\end{Rmk}

We are now ready to prove Conjecture \ref{ConjLCT-SEH} for $n=4$, following a similar strategy to that used in the $n=3$ case in \cite{GS}:

\begin{Th}
\label{ThConjDim4}
    Let $X$ be a $4$-dimensional analytic manifold and let $D \subset X$ be a free divisor in it. If $D$ satisfies LCT, then $D$ must be strongly Euler-homogeneous.
\end{Th}

\begin{proof}
    As these properties are local, we may assume $D=V(f)$ is a germ of free divisor in $(\C^4, 0)$ and it is enough to prove that if it satisfies LCT, then it is strongly Euler-homogeneous at $0$. 
    
    Note that, since $D$ is a product outside $D_0$, by Lemma \ref{LemProd}, it must be strongly Euler-homogeneous at those points (as we know LCT implies strong Euler-homogeneity for free divisors in dimension $3$). Let us suppose that it is not strongly Euler-homogeneous at $0$. In particular, it cannot be a product at $0$. 
    
    As LCT holds for $D$, by Theorem \ref{ThZeroTrace}, there must be at least one logarithmic derivation with non-zero trace (in particular, it cannot be topologically nilpotent). Thus, in the notation of Theorem \ref{ThFST}, $s \geq 1$ and we can find a formal coordinate system $x_1, \ldots, x_4$ and a basis of $\what{\Der}_f$, $\{\sigma, \delta_2, \delta_3, \delta_4\}$, such that $\sigma = \sum_{i=1}^4 \lambda_i x_i \partial_i$ with $\lambda_i \in \Q$, $\tr(\sigma) = \sum_{i=1}^4 \lambda_i \neq 0$ and $\delta_i$ is $\sigma$-homogeneous of degree $c_i \in \Q$ for $i=2,3,4$ (if $s>1$ so that some of the $\delta_i$ are diagonal, then they trivially commute with $\sigma$ and $c_i = 0$). Moreover, there exists a formal unit $u$ such that $g=uf$ verifies $\sigma(g) = \alpha_1 g$, with $\alpha_1 \in \C$. Let $\alpha_i \in \what{\hol}$ be such that $\delta_i(g) = \alpha_i g$ for $i=2,3,4$. As we are assuming $D$ is not strongly Euler-homogeneous at $0$, by Proposition \ref{PropForm}, $\alpha_i \in \what{\mm}$ for all $i = 1,\ldots,4$. In particular, $\alpha_1=0$.

    Now we distinguish cases according to the number of $\lambda_i$ that are zero:

    If every $\lambda_i \neq 0$, then each $x_i$ belongs to $J_1$ (recall that this is the ideal of $\what{\hol}$ generated by the entries of the formal Saito matrix). Moreover, this ideal is proper because $D$ is not a product at $0$ (and, thus, $\what{\Der}_g \subset \mm\what{\Der}$ by Proposition \ref{PropForm}), so $J_1 = \what{\mm}$, $I_1 = J_1^c = \mm$ and $D_0 = \{0\}$. Therefore, $D$ is strongly Euler-homogeneous outside $0$ but not at $0$, which contradicts Theorem \ref{ThConjOutsideDiscreteSet}.

    If three of the $\lambda_i$ are zero, say $\lambda_1=\lambda_2=\lambda_3=0$, then $\lambda_4 \neq 0$ and $\partial_4(g)=0$, so $D$ is (formally and then convergently by Proposition \ref{PropForm}) a product at $0$, reaching a contradiction.

    If exactly two of the $\lambda_i$ are zero, say $\lambda_1=\lambda_2=0$, $\lambda_3, \lambda_4 \neq 0$, then, as $\sigma(g)=0$, $g$ is of the form $\sum_{\lambda_3 j + \lambda_4 k = 0} a_{jk}(x_1, x_2) x_3^j x_4^k$. Equivalently, $g = \sum_{\mu \in \N} a_{\mu}(x_1, x_2) x_3^{\mu p} x_4^{\mu q}$ where $p, q \geq 1$ are coprime and $\lambda_3 p + \lambda_4 q = 0$ (if $p$ or $q$ vanishes, $g$ would be a product because it would not depend on all variables). By formal Saito's criterion (cf. \cite[Corollary 3.4]{dVR24}), $g$ is a unit times the determinant of the Saito matrix, which belongs to $(x_3, x_4)$ as $\lambda_1=\lambda_2=0$. This implies that $a_0(x_1, x_2) = 0$ and we can extract $x_3^p x_4^q$ as a common factor. But $g=uf$ is reduced (because so is $f$ in $\what{\hol}$ by Proposition \ref{PropForm}), so we must have $p=q=1$ and $\lambda_3 + \lambda_4 = 0$. Thus, $\tr(\sigma)=0$ and we get a contradiction.

    The only case left (and the hardest to discard) is that only one of the $\lambda_i$ is zero, say $\lambda_1=0$. Now, we have:

    \vspace{-0,2 cm}
    
    $$(x_2, x_3, x_4) \subset J_1 \subset \tilde{J}_1 \subset \sqrt{\tilde{J}_1} = \sqrt{(\alpha_2, \alpha_3, \alpha_4)}, $$

    \vspace{0,2 cm}

    \noindent where the last equality is due to Corollary \ref{CorSEHoutsideD0Formal}. As $\htt \sqrt{(\alpha_2, \alpha_3, \alpha_4)} = \htt (\alpha_2, \alpha_3, \alpha_4) \leq 3$, the inclusions cannot be strict. So we deduce $(x_2, x_3, x_4) = J_1 = \sqrt{(\alpha_2, \alpha_3, \alpha_4)}$.

    Since $\sigma(g) = 0$, $\sum_{i=1}^4 \lambda_i \beta_i = 0$ for any monomial $x^\beta$ appearing in $g$, so not all $\lambda_i$ can have the same sign (as $g$ is not a product, all variables must appear in some monomial of $g$). Let us assume without loss of generality that $\lambda_2 > 0$ and $\lambda_3, \lambda_4 < 0$.

    Let us also note that $\alpha_i$ is $\sigma$-homogeneous of degree $c_i$:

    \vspace{-0,2 cm}
    
    $$ c_i \alpha_i g = c_i \delta_i(g) = [\sigma, \delta_i](g) = \sigma(\alpha_i g) = \sigma(\alpha_i) g, $$

    \noindent so $\sigma(\alpha_i) = c_i \alpha_i$ for $i=2,3,4$.

    As there exists some $\beta' \geq 1$ such that $x_2^{\beta'} \in (\alpha_2, \alpha_3, \alpha_4)$, some $\alpha_i$, say $\alpha_2$, must have a monomial of the form $x_2^\beta$ for some $\beta \geq 1$ (otherwise, $(\alpha_2, \alpha_3, \alpha_4)$ would be contained in $(x_1, x_3, x_4)$ and clearly $x_2^{\beta'} \not \in (x_1, x_3, x_4)$). Then, $c_2 = \beta \lambda_2 > 0$.

    Let $A$ be the Saito matrix with respect to $\{\sigma, \delta_2, \delta_3, \delta_4\}$ and let $a_{ij} = \delta_i(x_j)$ (the entries of $A$ corresponding to the $\delta_i$). As $\delta_i$ is $\sigma$-homogeneous of degree $c_i$ and $x_j$ is $\sigma$-homogeneous of degree $\lambda_j$, it follows that $a_{ij}$ is $\sigma$-homogeneous of degree $c_i + \lambda_j$.

    Now, we claim: 

    \begin{itemize}
        \item Every $\sigma$-homogeneous formal power series $h$ of strictly positive degree (with respect to $\sigma$) belongs to $(x_2)$.

        \textit{Proof:} As $x_1$ has degree $\lambda_1 = 0$ and $x_3, x_4$ have negative degree ($\lambda_3$ and $\lambda_4$, respectively), a monomial in which $x_2$ does not appear must necessarily have a non-positive degree, so cannot be part of $h$.

        \item Every $\sigma$-homogeneous formal power series $h$ of degree zero belongs to $(x_2)+\C[[x_1]]$.

        \textit{Proof:} If $x_2$ does not appear in a monomial of degree zero, then neither $x_3$ nor $x_4$, having a strictly negative degree, can appear, so the monomial must be a power of $x_1$.
    \end{itemize}

    \vspace{0,3 cm}

    Since $g = v \cdot \det(A)$ for some formal unit $v$, it follows that $g \in (x_2, x_3, x_4)$. But then $g$ must be a multiple of $x_2$, since its degree is zero and a power of $x_1$ cannot be a monomial of $g$. This implies $\Der_g \subset \Der_{x_2} = \langle \partial_1, x_2 \partial_2, \partial_3, \partial_4 \rangle$ \cite[Lemma 3.4]{GS} and so $a_{i2}$ is a multiple of $x_2$ for all $i$. Thus, we can consider $\delta'_i = \delta_i - \frac{a_{i2}}{\lambda_2 x_2} \sigma$ that is still $\sigma$-homogeneous of degree $c_i$ (as $a_{i2}/x_2$ is $\sigma$-homogeneous of degree $c_i+\lambda_2-\lambda_2 = c_i$) and verifies $\delta'_i(g) = \alpha_i g$. It is clear that $\{\sigma, \delta'_2, \delta'_3, \delta'_4\}$ is a basis of $\Der_g$ whose Saito matrix has the same determinant that $A$. Note that $\delta_i'(x_2) = \delta_i(x_2)-\frac{a_{i2}}{\lambda_2 x_2} \sigma(x_2) = a_{i2}-a_{i2}=0$. For the sake of simplicity, this new basis will be denoted again by $\{\sigma, \delta_2, \delta_3, \delta_4\}$ and its Saito matrix by $A$. With this new notation, $a_{i2} = 0$ for $i=2,3,4$. Thus, expanding the determinant by the second column we get $g=v \cdot \det(A) = v \cdot \lambda_2 x_2 h$, where $h=\left| \begin{array}{ccc}
       a_{21} & a_{23} & a_{24} \\
       a_{31} & a_{33} & a_{34} \\
       a_{41} & a_{43} & a_{44} \\
    \end{array}\right|$.
    
    Note that $\det(A)$ is $\sigma$-homogeneous of degree $\sum_{i=2}^4 \lambda_i + \sum_{i=2}^4 c_i$ because so are all of its summands, and so the unit $v$ must also be $\sigma$-homogeneous. But a unit, having a non-zero constant term, can only be $\sigma$-homogeneous of degree $0$. Thus, we deduce that $\det(A)$ is also $\sigma$-homogeneous of degree $0$ and $\sum_{i=2}^4 \lambda_i + \sum_{i=2}^4 c_i = 0$.
    
    Since $g$ is reduced, $h$ cannot be a multiple of $x_2$. By the previous discussion, $x_2h$ is $\sigma$-homogeneous of degree $0$ and, thus, $h$ must be $\sigma$-homogeneous of degree $-\lambda_2$. Therefore, it has at least a monomial of the form $x_1^i x_3^j x_4^k$ with $j \lambda_3 + k \lambda_4 = -\lambda_2$ (or, equivalently, $\lambda_2 = j |\lambda_3| + k |\lambda_4|$).

    As $a_{21}$ and $\alpha_2$ are $\sigma$-homogeneous of degree $c_2 > 0$, by the first claim, they must be a multiple of $x_2$. Let us write $a_{21} = x_2 a'_{21}$ and $\alpha_2=x_2 \alpha'_2$. 

    Now consider $a_{23}$, which is $\sigma$-homogeneous of degree $c_2 + \lambda_3$ and belongs to $J_1 = (x_2, x_3, x_4)$. If $c_2 + \lambda_3 \geq 0$, then $a_{23}$ must also be a multiple of $x_2$ by the second claim. The same applies to $a_{24}$ if $c_2 + \lambda_4 \geq 0$. But $a_{23}$ and $a_{24}$ cannot be a multiple of $x_2$ at the same time, as this would imply that $h$ is also a multiple of $x_2$, so $c_2 + \lambda_3 < 0$ or $c_2 + \lambda_4 < 0$:

    \begin{itemize}
        \item If $c_2 + \lambda_3 = \beta \lambda_2 + \lambda_3 < 0$, then $|\lambda_3| > \beta \lambda_2$ and:

        $$ \lambda_2 = j|\lambda_3| + k |\lambda_4| > j\beta\lambda_2+k|\lambda_4| \geq j \lambda_2,$$

        so $j=0$ and $\lambda_2 = k |\lambda_4|$ (in particular, $k \neq 0$).

        \item If $c_2 + \lambda_4 = \beta \lambda_2 + \lambda_4 < 0$, then $|\lambda_4| > \beta \lambda_2$ and:

        $$ \lambda_2 = j|\lambda_3| + k |\lambda_4| > j |\lambda_3| + k\beta\lambda_2 \geq k \lambda_2,$$

        so $k=0$ and $\lambda_2 = j |\lambda_3|$ (in particular, $j \neq 0$).
        
    \end{itemize}

    Therefore, one of the degrees is negative and the other one is non-negative. Without loss of generality, we may assume that $c_2 + \lambda_3 < 0$ and $c_2 + \lambda_4 \geq 0$ (thus, $a_{24}$ is a multiple of $x_2$ by the previous discussion). Moreover, any monomial of $h$ in which $x_2$ does not appear is of the form $x_1^i x_4^k$ (recall that $j=0$) with $k$ fixed ($k=\lambda_2/|\lambda_4|$).

    We proceed in the same way with $x_3$: belonging to $\sqrt{(\alpha_2, \alpha_3, \alpha_4)}$, as before, there exists $\gamma \geq 1$ such that $x_3^\gamma$ is a monomial of some $\alpha_i$. It cannot be a monomial of $\alpha_2 = x_2 \alpha'_2$, so we may assume that it is a monomial of $\alpha_3$. Thus, $c_3 = \gamma \lambda_3$. 

    As $\sum_{i=2}^4 \lambda_i + \sum_{i=2}^4 c_i = 0$, we deduce:

    \vspace{-0,3 cm}

    \begin{align*}
        c_4 =& -\lambda_2-\lambda_3-\lambda_4-c_2-c_3 \\
            =& -(1+\beta)\lambda_2+(1+\gamma)|\lambda_3| + |\lambda_4| \\
            >& \hspace{0,1 cm} [-(1+\beta)+(1+\gamma)\beta]\lambda_2 + |\lambda_4| \\
            =& \hspace{0,1 cm} (\gamma \beta -1) \lambda_2 + |\lambda_4| \geq |\lambda_4| > 0,
    \end{align*}

    \noindent where we have used that $c_2 = \beta \lambda_2$, $c_3 = \gamma \lambda_3$, $\lambda_3, \lambda_4 < 0$ in the second equality, $|\lambda_3| > \beta \lambda_2$ in the following inequality and $\beta, \gamma \geq 1$ in the next one.

    But $c_4$ is the degree of $a_{41}$ and $\alpha_4$, so, again by the first claim, they are both a multiple of $x_2$: $a_{41} = x_2 a'_{41}$ and $\alpha_4 = x_2 \alpha'_4$. 

    Now recall that $g = v \cdot \det A = v \cdot \lambda_2 x_2 h$, so $\partial_2(g) = v \lambda_2 h + \lambda_2 x_2 \partial_2(v \cdot h)$. And if we apply Cramer's rule to the system $A \cdot (\partial_1(g), \partial_2(g), \partial_3(g),\partial_4(g))^t = (0, \alpha_2g, \alpha_3g,\alpha_4g)^t$ we get

    $$ \partial_2(g) = v \left| \begin{array}{cccc}
       0 & 0 & \lambda_3 x_3 & \lambda_4 x_4 \\
       a_{21} & \alpha_2 & a_{23} & a_{24} \\
       a_{31} & \alpha_3 & a_{33} & a_{34} \\
       a_{41} & \alpha_4 & a_{43} & a_{44} \\
    \end{array}\right| = v \left| \begin{array}{cccc}
       0 & 0 & \lambda_3 x_3 & \lambda_4 x_4 \\
       x_2 a'_{21} & x_2 \alpha'_2 & a_{23} &x_2 a'_{24} \\
       a_{31} & \alpha_3 & a_{33} & a_{34} \\
       x_2 a'_{41} & x_2 \alpha'_4 & a_{43} & a_{44} \\
    \end{array}\right|.$$
    
    Finally, evaluating the two expressions of $\partial_2(g)$ at $x_2=0$:

    \begin{align*}
        v|_{x_2=0} \cdot \lambda_2 h|_{x_2=0} =\partial_2(g)|_{x_2=0} =& v|_{x_2=0} \cdot \left| \begin{array}{cccc}
       0 & 0 & \lambda_3 x_3 & \lambda_4 x_4 \\
       x_2 a'_{21} & x_2 \alpha'_2 & a_{23} &x_2 a'_{24} \\
       a_{31} & \alpha_3 & a_{33} & a_{34} \\
       x_2 a'_{41} & x_2 \alpha'_4 & a_{43} & a_{44} \\
    \end{array}\right|_{x_2=0} = 0.
    \end{align*}

    As $v$ is a unit and $\lambda_2 \neq 0$, this implies $h|_{x_2=0}=0$. But then $h$ is a multiple of $x_2$, which is a contradiction with the fact that $g$ is reduced. 

    We conclude that $D$ must be strongly Euler-homogeneous at $0$.
    
\end{proof}

\begin{Rmk}
   In the proof of Theorem \ref{ThConjDim4}, we use a formal Saito's criterion for a convergent power series, which is an easy consequence of the original one but enough for our purpose. However, there is a purely formal Saito's criterion stating that, for a reduced $g \in \what{\hol}$, $\Der_g$ is a free $\what{\hol}$-module of rank $n$ if and only if $g$ is, up to a unit, the determinant of the formal Saito matrix with respect to some elements $\delta_1, \ldots, \delta_n \in \Der_g$. 
   
   One implication of this result is given in \cite[Proposition 4.2]{GS}. The converse is unwritten but it is a consequence of the Hilbert-Burch theorem \cite[Theorem 20.15]{EisenbudComAlg} that was hinted by R. O. Buchweitz.
   
\end{Rmk}

\section{The case of linear free divisors}
\label{SectionLFD}

Recall that linear free divisors are those free divisors $D \subset \C^n$ that have a global basis of linear logarithmic derivations (i.e. derivations that only have linear part). By Saito's criterion, these divisors are always defined by homogeneous polynomials of degree $n$, so there is always a strong Euler vector field at the origin: $\chi = \sum_{i=1}^n x_i \partial_i$. In \cite[Theorem 7.10]{GMNS}, Granger-Mond-Nieto-Schulze prove:

\begin{Th}
Every linear free divisor in dimension $n \leq 4$ is locally
quasihomogeneous and hence LCT and GLCT (Global Logarithmic Comparison Theorem) hold.

\end{Th}

They conjecture that this is true in any dimension. However, the following is an example of a linear free divisor in dimension $5$ such that it is not even strongly Euler-homogeneous and that does not verify LCT:

\begin{Ex}
\label{ExLFDnotLCTnorSEH}
Let $D=V(f) \subset \C^5$ with $f = x(8x^3 u-y (8t x^2-4xyz+y^3))$. \textsc{Macaulay2} \cite{M2} tells us that $D$ is a linear free divisor and that a basis of the module of logarithmic derivations is formed by: 

\vspace{0,1 cm}

$$\begin{array}{l}

    \delta_1 = x \partial_x + y \partial_y + z \partial_z + t \partial_t + u \partial_u, \\

    \delta_2 = -4x \partial_x + y \partial_y + 6z \partial_z + 11t \partial_t + 16u \partial_u, \\

    \delta_3 = x \partial_y + y \partial_z + z \partial_t + t \partial_u, \\

    \delta_4 = 2x \partial_z + y \partial_t, \\
    
    \delta_5 = x \partial_t + y \partial_u.

\end{array}$$

Here, $\delta_1$ is the Euler vector field ($\delta_1(f) = 5f$) and $\delta_i(f) = 0$ for $i=2, \ldots, 5$. \\

The extended Saito matrix is then:

$$\tilde{A} = \left( \begin{array}{ccccc|c}
    x & y & z & t & u & -5 \\
    -4x & y & 6z & 11t & 16u & 0 \\
    0 & x & y & z & t & 0 \\
    0 & 0 & 2x & y & 0 & 0 \\
    0 & 0 & 0 & x & y & 0 \\
\end{array} \right). $$

\vspace{0,2 cm}

We have that:

\vspace{0,1 cm}

$$\begin{array}{l}

    \tilde{D}_0 = \varnothing, \\

    \tilde{D}_1 = D_0 = \{0\}, \\
    
    \tilde{D}_2 = D_1 = \{x=y=z=t=0\}, \\

    \tilde{D}_3 = \{x=y=0\}, \hspace{0,2 cm} D_2 = \{x=y=z=0\} \cup \{x=y=t^2-2uz=0\}. \\

\end{array}$$

\vspace{0,3 cm}

By Theorem \ref{ThGeomEH}, $D$ is strongly Euler-homogeneous if and only if $\tilde{D}_0 = \varnothing$ and $\tilde{D}_i=D_{i-1}$ for $i=1, 2, 3$, so we conclude that $D$ is not strongly Euler-homogeneous at any point of $\tilde{D}_3 \setminus D_2$. However, $\dim D_i = i$ for $i=0, 1, 2$, so $D$ is a Koszul-free divisor (Theorem \ref{ThCaractSH}). Since every Koszul-free divisor satisfying LCT must be strongly Euler-homogeneous, this divisor cannot satisfy LCT.

Moreover, its $b$-function is 

\vspace{-0,3 cm}

$$b(s) = (s+1)^3(2s+1)^2(3s+2)(3s+4)(4s+3)(4s+5)(6s+5)(6s+7).$$

In \cite[Theorem 1.4 and Conjecture 1.5]{SymmetryGS}, M. Granger and M. Schulze proved that all reductive linear free divisors have $b$-functions with symmetric roots about $-1$ and conjecture that this is true for any linear free divisor. However, this $b$-function is not symmetric about $-1$ due to the factor $(2s+1)$. We deduce that $D$ is not reductive and that this conjecture is also false for non-reductive linear free divisors.
\end{Ex}

Nevertheless, an immediate consequence of Theorem \ref{ThConjDim4} is that Conjecture \ref{ConjLCT-SEH} holds for linear free divisors in $n=5$:

\begin{Cor}
    Let $D$ be a linear free divisor in $\C^5$. If $D$ satisfies LCT, then $D$ is strongly Euler-homogeneous.
\end{Cor}

\begin{proof}
    Since we know the result is true for any free divisor in dimension $4$, we have that $D$ is strongly Euler-homogeneous outside $D_0$, where it is a product. But $D$ is linear, so $D_0$ is only the origin, where it is clearly strongly Euler-homogeneous. Thus, $D$ is strongly Euler-homogeneous. 
    
\end{proof}

\begin{Rmk}
    In fact, the same proof is valid to argue that if Conjecture \ref{ConjLCT-SEH} is true in dimension $n$, then it is true for linear free divisors in dimension $n+1$.
\end{Rmk}

\section*{Acknowledgements}
The author would like to thank his PhD advisors Luis Narváez Macarro and Alberto Castaño Domínguez for their support and guidance while writing this paper. He also thanks Dan Bath for his insightful comments and the referee for their careful reading and valuable suggestions, which significantly improved this article.

\addcontentsline{toc}{section}{Bibliography}

\bibliographystyle{plainurl}
\bibliography{biblio}
\thispagestyle{plain}

\end{document}